\documentclass[11pt,reqno]{amsart}
\usepackage[colorlinks=true,allcolors=blue,backref=page]{hyperref}

\usepackage{latexsym, amsfonts, amsmath, amssymb, amsthm, cite}

\usepackage{stmaryrd}
\usepackage{fullpage}

\newtheorem{theorem}{Theorem}[section]
\newtheorem{lemma}[theorem]{Lemma}
\newtheorem{proposition}[theorem]{Proposition}

\renewcommand{\leq}{\leqslant}
\renewcommand{\geq}{\geqslant}

\usepackage{graphicx}
\usepackage{mathrsfs}
\numberwithin{equation}{section}

\renewcommand{\leq}{\leqslant}
\renewcommand{\geq}{\geqslant}

\newcommand{\md}[1]{\ensuremath{(\operatorname{mod}\, #1)}}

\newcommand\E{\mathbf{E}}

\newcommand\C{\mathbf{C}}

\newcommand\Q{\mathbf{Q}}

\renewcommand\Im{\operatorname{Im}} 
\renewcommand\Re{\operatorname{Re}} 

\newcommand\ns{n^{\flat}} 
\newcommand\nl{n^{\sharp}} 
\newcommand\ldd{L_{d,\tilde d}} 
\newcommand\dmax{\Delta}

\newcommand\eps{\varepsilon}

\allowdisplaybreaks[1]

\begin{document}

\title[Covering integers by $x^2 + dy^2$]{Covering integers by $x^2 + dy^2$}



\begin{abstract}
What proportion of integers $n \leq N$ may be expressed as $x^2 + dy^2$ for some $d \leq \dmax$, with $x,y $ integers? Writing $\dmax = (\log N)^{\log 2} 2^{\alpha \sqrt{\log \log N}}$ for some $\alpha \in (-\infty, \infty)$, we show that the answer is $\Phi(\alpha) + o(1)$, where $\Phi$ is the Gaussian distribution function $\Phi(\alpha) = \frac{1}{\sqrt{2\pi}} \int^{\alpha}_{-\infty} e^{-x^2/2} dx$.

A consequence of this is a phase transition: almost none of the integers $n \leq N$ can be represented by $x^2 + dy^2$ with $d \leq (\log N)^{\log 2 - \eps}$, but almost all of them can be represented by $x^2 + dy^2$ with $d \leq (\log N)^{\log 2 + \eps}$. \end{abstract}

\author{Ben Green}
\address{Mathematical Institute\\
Radcliffe Observatory Quarter\\
Woodstock Road\\
Oxford OX2 6GG\\
England}
\email{ben.green@maths.ox.ac.uk}

\author{Kannan Soundararajan}
\address{Department of Mathematics\\ Stanford University \\ Stanford CA 94305}
\email{ksound@stanford.edu}

\maketitle

\setcounter{tocdepth}{1}
\tableofcontents

\section{Introduction}

In this paper we are interested in how many integers $\leq N$ are covered by the values taken by the quadratic forms $x^2 + dy^2$, $d \leq \Delta$. 
Our main result is the following, which gives a fairly complete answer to this question.

\begin{theorem}[Main theorem]\label{mainthm}
Let $N$ be large and write, for some real number $\alpha$, \[ \dmax = (\log N)^{\log 2} 2^{\alpha \sqrt{\log \log N}}.\]
Then 
$$
\# \{ n\leq N: n = x^2 +dy^2 \text{ for some } 1\leq d \leq \Delta \}  = (\Phi(\alpha) + o(1)) N, 
$$ 
where $\Phi$ is the Gaussian distribution function $\Phi(\alpha) = \frac{1}{\sqrt{2\pi}}  \int^{\alpha}_{-\infty} e^{-x^2/2} dx$.
\end{theorem}

The problem of covering integers by this family of binary quadratic forms seems to have been first considered in the work of Hanson and Vaughan 
\cite{hv}.   Using the circle method they established that almost all integers $n\leq N$ may be covered with $\Delta = \log N (\log \log N)^{3+\eps}$ 
for any $\eps >0$, and that a positive proportion of the integers below $N$ may be covered using $\Delta = \log N \log \log N$.  Diao \cite{diao} found a 
much shorter proof of the latter result, and in his argument $d$ could be restricted to prime values so that a smaller set of forms is used.  

Landau established that the number of integers below $N$ that are sums of two squares is 
$\sim BN/(\log N)^{1/2}$ for a positive constant $B$.  This 
was extended by Bernays to show that for any fixed primitive positive definite binary quadratic form $f$, the number of integers below $N$ 
that are represented by $f$ is $\sim B_f N (\log N)^{-1/2}$, for a positive constant $B_f$ (which in fact depends only on the discriminant of $f$).   
More recently, Blomer \cite{blomer1, blomer2}, and Blomer and Granville \cite{bg} consider in detail the number of integers up to $N$ that are 
represented by $f$ uniformly in the form $f$ (thus allowing the discriminant to grow with $N$).   These results, taken with the union bound, suggest that if $\Delta$ 
is smaller than $(\log N)^{1/2-\eps}$ then almost all $n\leq N$ cannot be covered by the forms $x^2+dy^2$ with $d\leq \Delta$.  
However, as Theorem \ref{mainthm} reveals, the true threshold for $\Delta$ is neither $(\log N)^{1/2}$ nor $\log N$ but instead $(\log N)^{\log 2}$.

We shall in fact prove a more precise version of Theorem \ref{mainthm}, counting the number of integers below $N$ with $k$ prime factors that may be represented 
as $x^2+ dy^2$ with $d\leq \Delta$.   Throughout let $\Omega(n)$ denote the number of prime factors of $n$ counted with multiplicity, and define 
\[
{\mathcal A}(N,k) = \{ n \leq N: \; \Omega(n) =k\}.
\] 
Recall that most integers below $N$ have about $\log \log N$ prime factors, a result first established by Hardy and Ramanujan.  The well known work of 
Erd{\H o}s and Kac established that $\Omega(n)$ has a normal distribution with mean $\sim \log \log N$ and variance $\sim \log \log N$, while Selberg's work \cite{selberg-54} gave 
still more precise results establishing an asymptotic formula for ${\mathcal A}(N,k)$ uniformly in a wide range of $k$. To reduce the visual complexity of expressions involving double logs later on, it is convenient to set (throughout the paper)
\[ k_0 := \log \log N.\]
The following simplified version of Selberg's result is an immediate consequence of \cite[Theorem II.6.5]{tenenbaum}.

\begin{lemma} \label{lem1.2} Let $N$ be large.  Uniformly for integers $k$ in the range $|k -k_0| \leq \tfrac 12 k_0$,
we have 
\[
|{\mathcal A}(N,k)| = \frac{N}{\log N} \frac{k_0^k}{k!}\Big( 1+ O\Big( \frac{1+|k-k_0|}{k_0}\Big)\Big). 
\]
\end{lemma} 

For a given $k$ in a suitable interval around $\log \log N$, we shall show (the `upper bound', Theorem \ref{thm1.3} below) that almost none of the integers in ${\mathcal A}(N,k)$ are represented by $x^2+dy^2$ 
with $d\leq \Delta$ if $\Delta$ is a bit smaller than $2^k$.  This changes when $\Delta$ becomes a bit larger than $2^k$, when almost all the integers in ${\mathcal A}(N,k)$ may be so represented. This is the `lower bound', Theorem \ref{thm1.4} below. From these results, Theorem \ref{mainthm} will follow swiftly.

We turn now to the precise statements.

\begin{theorem}[Upper bound] \label{thm1.3}  Let $N$ be large, and let $k$ be an integer in the range 
\begin{equation} 
\label{1.3} 
|k- k_0| \leq k_0^{2/3}. 
\end{equation}
Suppose $\Delta \leq 2^k/k^4$.  The number of integers $n \in {\mathcal A}(N,k)$ that may be written as $x^2 + dy^2$ with $1\leq d\leq \Delta$ is $\ll N/k_0$.
\end{theorem}

An application of Stirling's formula (see \eqref{stirling-3} below) shows that for $k$ in the range \eqref{1.3} 
$$ 
|{\mathcal A}(N,k)| = \frac{N}{\sqrt{2\pi k_0}} \exp\Big( -\frac{(k-k_0)^2}{2k_0}\Big) \Big( 1+ O\Big( k_0^{-1/5}\Big) \Big). 
$$ 
Thus Theorem \ref{thm1.3} is really of interest only when $|k-k_0| \leq (k_0 \log k_0)^{1/2}$.  This range still includes 
most typical integers below $N$, and  Theorem \ref{thm1.3} may be used to establish the upper bound for the integers below $N$ of the form $x^2+dy^2$ with $d\leq \Delta$ that is implicit in Theorem \ref{mainthm}.  The corresponding lower bound in Theorem \ref{mainthm} is implied by the following result. 

\begin{theorem}[Lower bound] \label{thm1.4}  Let $N$ be large, and let $k$ be an integer in the range given in \eqref{1.3}.   Suppose $\Delta \geq k^3 2^k$.  Let ${\mathcal E}(N,k)$ 
denote the set of integers in ${\mathcal A}(N,k)$ that cannot be represented as $x^2+dy^2$ with $d\leq \Delta$.  Then  
\[ 
|{\mathcal E}(N,k)| \ll \frac{|\mathcal{A}(N,k)|}{\log k_0} + Nk_0^{-3/4}.\]
\end{theorem} 

Similarly to Theorem \ref{thm1.3}, Theorem \ref{thm1.4} is really of interest only in the range \[ |k- k_0| \leq (\tfrac{1}{2} k_0 \log k_0)^{1/2},\] but this range includes most typical integers 
below $N$.   In Section 2 we shall deduce our main result Theorem \ref{mainthm} from Theorems \ref{thm1.3} and \ref{thm1.4}.  

Since our main interest is in establishing Theorem \ref{mainthm} we have made no effort to optimize the error terms and ranges for $k$ in Theorems \ref{thm1.3} 
and \ref{thm1.4}.  It would be of interest to establish analogues of these results uniformly in a wide range of $k$ (although when $k$ is large it may be better to work with 
$\omega(n)=k$, where $\omega(n)$ counts the number of distinct prime factors of $n$).   One case of particular interest may be $k=1$: representing primes up to $N$ 
using the quadratic forms $x^2 +dy^2$ with $d\leq \Delta$.  Here it would be possible to establish that a proportion $\rho(\Delta)$ of the primes up to $N$ may be so
represented with
 $\rho(1)=1/2$ (by Fermat's result on representing primes of the form $1 \ \md 4$ as a sum of two squares), $\rho(\Delta) <1$ for all $\Delta$, 
and $\rho(\Delta) \to 1$ as $\Delta \to \infty$.  Determining $\rho(\Delta)$ precisely, or understanding its precise asymptotic behavior as $\Delta$ gets large, seem like challenging and 
delicate problems.

Let us indicate very briefly the ideas behind Theorems \ref{thm1.3} and \ref{thm1.4}; here and in the rest of the introduction we shall be a little informal, and also assume that the reader is familiar with the classical theory of binary quadratic forms (which will be recalled in Section 3).  Recall that a square-free integer $n$ may be represented by some binary quadratic form of negative discriminant $D$ if and only if $\chi_D(p)= 1$ for all primes $p$ dividing $n$ (assume that $n$ is coprime to $D$).  If $n$ has $k$ prime factors, then each condition 
$\chi_D(p)=1$ has a $50\%$ chance of occurring, so that $n$ may be represented by some binary quadratic form of discriminant $D$ with probability $2^{-k}$.  This suggests 
that $\Delta$ must be about size $2^k$ in order to have a chance of representing many integers with $k$ prime factors.   This is the idea behind Theorem \ref{thm1.3}, and it 
can be made precise without too much difficulty (see Section 4).

The more difficult part of our argument is Theorem \ref{thm1.4}, which constitutes the bulk of the paper.  
If $\Delta$ is substantially larger than $2^k$, then the heuristic that we just mentioned would suggest that for most integers $n\leq N$ with $k$ prime factors there 
would be some negative discriminant $D$ with $|D|\leq \Delta$ such that $n$ is representable by some binary quadratic form of discriminant $D$, and indeed there would be a total of about $2^k$ such representations of $n$.  The number of inequivalent classes of binary quadratic forms of discriminant $D$ is the class number, which is of size $|D|^{1/2+ o(1)}$.   It is therefore likely that some of the $2^k$ (which is about $\Delta$) representations of $n$ would come from the principal form $x^2+dy^2$ (corresponding to the discriminant $D=-4d$), and indeed that there should be about $2^k/|D|^{1/2+o(1)}$ representations of $n$ by $x^2+dy^2$.  We make this heuristic precise by using  class group characters and their associated $L$-functions, together with a second moment method.   It would be relatively straightforward to obtain a version of Theorem \ref{thm1.4} where a positive proportion of the elements in ${\mathcal A}(N,k)$ are represented by the forms $x^2+dy^2$ with $d\leq \Delta$.  However, it is more delicate to obtain almost all integers in ${\mathcal A}(N,k)$, and to achieve this we impose congruence conditions on $d$ for all primes $p$ below a slowly growing parameter $W$.  A key fact is that when discriminants $d$ are restricted to such progressions, the value of $L(1,\chi_d)$ remains more or less constant.   To simplify genus theory considerations, we further restrict 
attention to prime values of $d$, but this is merely a matter of convenience.  

For $k$ sufficiently close to $k_0$, Theorems \ref{thm1.3} and \ref{thm1.4} show that the number of represented elements in ${\mathcal A}(N,k)$ undergoes 
a rapid phase transition as one goes from $\Delta = 2^k/k^4$ (when $0\%$ of ${\mathcal A}(N,k)$ is covered) to $\Delta = 2^k k^3$ (when $100\%$ of ${\mathcal A}(N,k)$ is 
covered).  While there is some scope to narrow the gap between $2^k/k^4$ and $2^k k^3$, the restriction to prime values of $d$ in our proof of Theorem \ref{thm1.4} would prevent us from fully closing this gap.   It seems likely that a more precise cutoff phenomenon occurs:  when $\Delta = \beta \sqrt{k} 2^k$ there is a proportion $p(\beta)$ of integers in 
${\mathcal A}(N,k)$ that are represented, with $0< p(\beta) <1$ for all $0 < \beta < \infty$, and with $p(\beta) \to 0$ as $\beta \to 0$ and $p(\beta) \to 1$ as $\beta \to \infty$.  Possibly our arguments, together with additional ideas taking into account genus theory, could be used to establish part of this cutoff phenomenon, and we hope that an interested reader will take up the challenge.  
 
Our discussion so far has been confined to representing almost all integers below $N$ using the forms $x^2+dy^2$ with $d\leq \Delta$.  It is natural to ask what happens if {\sl all} 
 integers below $N$ are to be represented.   Taking $x= \lfloor \sqrt{n}\rfloor$ and $y=1$ we see that $\Delta = 2\sqrt{N}$ suffices, and going beyond this trivial bound 
 already seems an interesting problem.  Since integers below $N$ have $\ll \log N/\log \log N$ distinct prime factors, extrapolating Theorem \ref{thm1.4} we may expect that 
 $\Delta =\exp( C \log N/\log \log N)$ is sufficient for some constant $C>0$.  As evidence towards this conjecture, we note that progress can be made in two weaker versions.  
 By a simple application of the pigeonhole principle one can show that every positive integer below $N$ may be represented by some non-degenerate binary quadratic form $f$ with $|\text{disc}(f)| \leq \exp(C \log N/\log \log N)$ with $C$ being any constant larger than $\log 4$.  Here non-degenerate means that the quadratic form does not factor into linear forms, or equivalently that the discriminant is not a square.   In fact, all elements of ${\mathcal A}(N,k)$ can be represented by some non-degenerate binary quadratic form 
 with absolute discriminant $\ll 4^k$ (for instance, all primes are of the form $x^2+y^2$, $x^2+ 2y^2$ or $x^2-2y^2$).  
The pigeonhole argument does not allow us to restrict attention to positive definite forms (although one can restrict attention to indefinite forms), let alone the smaller family of principal positive definite forms.  Assuming GRH for quadratic Dirichlet $L$-functions it can be shown that all integers below $N$ may be represented by some positive definite binary quadratic form with absolute discriminant below $\exp(C \log N/\log \log N)$ for any $C> \log 4$, and indeed that all elements in ${\mathcal A}(N,k)$ can be represented by such forms with absolute discriminant $\ll 4^k (\log N)^4$.
 
In the other direction, we may ask how large must $\Delta$ necessarily be if all integers $n\leq N$ are represented as $x^2 +dy^2$ with $d\leq \Delta$.   Complementing our 
 discussion above, we can establish here that $\Delta$ must be at least $\Delta_0= \exp(c \log N/\log \log N)$ for a positive constant $c$.  In fact, we can 
 establish the stronger result that there exists a square-free integer $n\leq N$ such that for any fundamental discriminant $d$ with $1 < |d|\leq \Delta_0$ 
 there exists a prime factor $p$ of $n$ with $\chi_d(p)=-1$.  Such an integer $n$ cannot be represented by any primitive non-degenerate binary quadratic form 
 with absolute discriminant below $\Delta_0$.  This result, which may be viewed as a variant of the least quadratic non-residue problem, follows from an application of log-free zero density estimates; details will be supplied elsewhere.

 Lastly, we draw attention to three papers from the literature where related problems concerning the integers represented by a family of binary quadratic forms are considered:  Blomer's work on sums of two squareful numbers \cite{blomer1}, the work of Bourgain and Fuchs \cite{BF} on Apollonian circle packings, and the work of Ghosh and Sarnak \cite{GhSa} on Markoff type cubic curfaces.  \vspace*{8pt} 
 
\emph{Notation.} For the most part notation will be introduced when it is needed. However, we remind the reader that $k_0$ will always denote $\log \log N$. From Section \ref{sec5} onwards, $W$ will denote a quantity which tends to infinity with $N$ sufficiently slowly; we will take $W := \log \log \log N$ for definiteness.\vspace*{8pt}

\emph{Plan of the paper.} Section \ref{sec2} is devoted to the proof that the upper and lower bounds (Theorems \ref{thm1.3} and \ref{thm1.4} respectively) imply the main theorem, Theorem \ref{mainthm}. Section \ref{sec3} gives some standard background on binary quadratic forms which will be used throughout the rest of the paper. In Section \ref{sec4}, we prove the relatively straightforward upper bound, Theorem \ref{thm1.3}.

The remainder of the paper is devoted to the much more involved proof of the lower bound, Theorem \ref{thm1.4}. First, we formulate a more technical variant of this result, Theorem \ref{thm5.1}. This result allows us to restrict attention to representing integers not divisible by $4$, using only quadratic forms $x^2 + dy^2$ with $d$ ranging over primes in certain congruence classes. The deduction of Theorem \ref{thm1.4} from Theorem \ref{thm5.1} is short and is given immediately after the statement of the latter. 

The proof of Theorem \ref{thm5.1} is via the second moment method. We divide the computations that arise into four separate technical propositions, Propositions \ref{prop5.2}, \ref{prop5.3}, \ref{prop5.4} and \ref{prop5.5}. The synthesis of these propositions to give a proof of Theorem \ref{thm5.1} is accomplished in Section \ref{sec6}. 

The final sections of the main part of the paper are devoted to the proofs of these four technical propositions. Proposition \ref{prop5.5} is a statement about averages of certain $L(1,\chi)$ and we handle it first, in Section \ref{sec7}. The remaining three results all require some background on class group $L$-functions, and Section \ref{sec8-class-group} provides an overview and references for the necessary material. Finally, the proofs of Propositions \ref{prop5.2}, \ref{prop5.3} and \ref{prop5.4} are given in Sections \ref{sec9}, \ref{sec10} and \ref{sec11} respectively.

Sections \ref{sec8-class-group}, \ref{sec9} and \ref{sec10} use Selberg's techniques \cite{selberg-54}. There is no particularly convenient reference for what we require, so we provide full details. The more standard parts of this may be found in Appendix \ref{app-A}.\vspace*{8pt}

\emph{Acknowledgments.} This work began at the 2022 Oberwolfach Analytic Number Theory meeting, and it is a pleasure to thank the Mathematisches Forschungsinstitut Oberwolfach for the stimulating working conditions.   BG is supported by a Simons Investigator grant and is grateful to the Simons Foundation for their continued support.  KS is supported in part by a Simons Investigator 
award from the Simons foundation, and a grant from the National Science Foundation.

\section{The upper and lower bounds imply the main theorem} \label{sec2}

In this section we show how Theorem \ref{mainthm} follows from Theorems \ref{thm1.3} and \ref{thm1.4}.

Suppose, as in the statement of Theorem \ref{mainthm}, that $\dmax = (\log N)^{\log 2} 2^{\alpha \sqrt{\log \log N}}$. It is enough to prove the result for 
\begin{equation}\label{alpha-bdd} |\alpha |\leq (\log \log N)^{1/10};
\end{equation} 
the result for all $\alpha$ follows from this case and the fact that \[ \Phi(-(\log \log N)^{1/10}) = o(1), \quad \Phi((\log \log N)^{1/10}) = 1 - o(1).\] 
Suppose henceforth that \eqref{alpha-bdd} holds.

Let $k^-$ be defined as the solution to $\Delta = k^3 2^k$, and $k^+$ as the solution to $\Delta = 2^k/k^4$.  Then, one may check that   
\[ 
k^+,   k^- =  \frac{\log\Delta}{\log 2} + O(\log \log \Delta)  =  k_0 + \alpha \sqrt{k_0} + O(\log k_0).
\]
In particular, by the assumption \eqref{alpha-bdd}, we see that $|k^{\pm} - k_0| \leq 2 k_0^{3/5}$.

For $k$ in the range $k_0-2 k_0^{3/5} \leq k \leq k^-$, Theorem \ref{thm1.4} shows that the number of integers in ${\mathcal A}(N,k)$ 
that may be represented as $x^2+dy^2$ with $d\leq \Delta$ is 
\begin{equation} 
\label{2.3} 
|{\mathcal A}(N,k)| \Big( 1+ O\Big(\frac{1}{\log k_0}\Big) \Big) + O\big(N k_0^{-3/4}\big). 
\end{equation} 
Stirling's formula 

and the approximation $1 -x = \exp(-x -\frac{x^2}{2} + O(x^3))$ with $x = 1 - \frac{k_0}{k}$ ($ = O(k_0^{-2/5})$) shows in this range of $k$ that 
\begin{equation}\label{stirling-3}
\frac{k_0^k}{k!} = \frac{1}{\sqrt{2\pi k_0}} \exp\Big( k_0 - \frac{(k_0-k)^2}{2k_0} + O\big( k_0^{-1/5}\big) \Big), 
\end{equation}
so that using Lemma \ref{lem1.2}, we may see that the quantity in \eqref{2.3} is 
$$
 \frac{N}{\sqrt{2\pi k_0} } \exp\Big( -\frac{(k-k_0)^2}{2k_0}\Big) \Big(1 + O\Big( \frac{1}{\log k_0} \Big) \Big) + O\big(N k_0^{-3/4}\big). 
$$
Summing over all $k$ in this range, we conclude that the number of integers $n\leq N$ that may be written as $x^2+dy^2$ with $d\leq \Delta$ is at least 
$$ 
\frac{N}{\sqrt{2\pi k_0}} \sum_{k_0 - 2k_0^{3/5} \leq k \leq k^- } \exp\Big( - \frac{(k-k_0)^2}{2k_0}\Big) + O\Big(\frac{N}{\log k_0} \Big) 
= (\Phi(\alpha) +o(1)) N, 
$$ 
upon approximating the sum by the corresponding integral.  This shows the lower bound implicit in Theorem \ref{mainthm}.  

To obtain the corresponding upper bound, 
note that for $k$ in the range $k^+ \leq k \leq k_0+2k_0^{3/5}$, Theorem \ref{thm1.3} shows that the number of integers in ${\mathcal A}(N,k)$ 
that cannot be represented as $x^2+dy^2$ with $d\leq \Delta$ is 
$|{\mathcal A}(N,k)| + O(N/k_0)$.  Using Lemma \ref{lem1.2} and Stirling's formula as above, this is 
$$ 
 \frac{N}{\sqrt{2\pi k_0} } \exp\Big( -\frac{(k-k_0)^2}{2k_0}\Big) \Big(1 + O\big( k_0^{-1/5} \big) \Big) + O\big(Nk_0^{-1}\big). 
$$
Summing over all $k$ in this range, we conclude the number of integers up to $N$ that cannot be represented as 
$x^2+dy^2$ with $d\leq \Delta$ is at least 
$$ 
\frac{N}{\sqrt{2\pi k_0}} \sum_{k^+ \leq k \leq k_0 +2k_0^{3/5} } \exp\Big( - \frac{(k-k_0)^2}{2k_0}\Big) + O\big(N k_0^{-1/5} \big) 
= (1-\Phi(\alpha) +o(1)) N. 
$$ 
This implies the upper bound implicit in Theorem \ref{mainthm}, and completes the proof. 

\section{Background on quadratic forms} \label{sec3}

For the theory in the rest of this section, good resources are \cite{cox}, \cite{davenport},  \cite[Chapter 22]{IwKo}, or 
\cite{zagier}.

\subsection{Fundamental discriminants and characters} 
A fundamental discriminant is an integer $D$ of the following type:  either $D \equiv 1 \ \md{4}$ and squarefree, or 
(ii) $D = 4m$ with $m \equiv 2,3 \  \md{4}$ and $m$ squarefree.   Apart from $D = 1$, these are precisely the discriminants of quadratic fields over $\Q$, and indeed the discriminant of $\Q(\sqrt{D})$ is $D$.  Equivalently, if $m$ is square-free, the quadratic field $\Q(\sqrt{m})$ 
has discriminant $4m$ if $m \equiv 2,3 \ \md{4}$ and $m$ if $m \equiv 1\  \md{4}$.  

Associated to the fundamental discriminant $D$ is the primitive quadratic Dirichlet character $\chi_{D}(n) = (\frac {D}{n})$, where the symbol here is the Kronecker symbol. 
This is defined to be completely multiplicative, and specified on the primes by the following:
\begin{itemize}
\item If $p$ is an odd prime, $\chi_{D}(p)= (\frac{D}{p})$ is the Legendre symbol;
\item $\chi_{D}(2) = 0$ if $D \equiv 0 \ \md{4}$, $1$ if $D \equiv 1 \ \md{8}$ and $-1$ if $D \equiv 5\  \md{8}$;
\item $\chi_{D}(-1) = \mbox{sgn}(D)$.
\end{itemize}

The Kronecker symbol $\chi_{D}$ is a primitive character of modulus $|D|$. It describes the splitting type of a prime $p$ in the quadratic field $K=\Q(\sqrt{D})$:  a prime $p$ splits in $\Q(\sqrt{D})$ if $(\frac{D}{p})=1$, remains inert if $(\frac{D}{p}) =-1$, and ramifies when  $(\frac{D}{p})=0$.  Thus the Dedekind zeta-function of the field $K$ is given by 
$$ 
\zeta_K(s) = \sum_{\mathfrak{a} \neq 0} (N \mathfrak{a})^{-s} = \zeta(s) L(s,\chi_D), 
$$ 
and the number of ideals in $\mathcal{O}_K$ of norm $n$ is 
$$ 
(1 \ast \chi_{D})(n) = \sum_{\ell | n} \chi_{D}(\ell). 
$$

\subsection{Positive definite forms and imaginary quadratic fields} \label{sec3.2} 

Let $D < 0$ be a negative fundamental discriminant, and let $K= \Q(\sqrt{D})$ denote the corresponding imaginary quadratic field.  

There is a well-known correspondence (going back to Gauss) between ideal classes in $K$ and equivalence classes of positive definite binary quadratic forms of discriminant $D$.  In particular, principal ideals in $K$ are in correspondence with the principal binary quadratic form given by $x^2 + \frac {|D|}4 y^2$ (in the case $D\equiv 0 \ \md 4$) and $x^2 + xy + \frac{1+|D|}{4} y^2$ (in the case $D \equiv 1 \ \md 4$, so that $|D|=-D \equiv 3 \ \md 4$).  

A key object of interest for us is 
\[
R_D (n) = \# \{ {\mathfrak a}:  N(\mathfrak a) =n, \; {\mathfrak a} \text{ principal } \} 
\]
which counts the number of principal ideals in ${\mathcal O}_K$ of norm $n$.  If $D\equiv 0 \ \md 4$, then a principal ideal ${\mathfrak a}$ of norm $n$ may be written as 
$(a+b\sqrt{D/4})$ and corresponds to two representations of $n$ by the principal form $x^2+ \frac{|D|}{4} y^2$, namely $n= (\pm a)^2 + \frac{|D|}{4} (\pm b)^2$ (with the exception of $D=-4$, where it corresponds to $4$ representations by the principal form $x^2+y^2$).   Similarly if $D\equiv 1\ \md 4$, a principal ideal ${\mathfrak a}$ of norm $n$ may be written as $(a + b \frac{1+\sqrt{D}}{2})$ and corresponds to two representations of $n$ by the principal form $x^2 +xy + \frac{1+|D|}{4} y^2$ (with the exception of $D=-3$, where it corresponds to $6$ representations by the principal form $x^2+xy+y^2$).  

We remark that 
\begin{equation} 
\label{3.2} R_D(n) \leq (1*\chi_D)(n),
\end{equation} 
since $(1*\chi_D)(n)$ counts all ideals with norm $n$, and that each ideal of norm $n$ corresponds to two (or $4$ when $D=-4$, or $6$ when $D=-3$) representations of $n$ by some equivalence class of binary quadratic forms of discriminant $D$.

To isolate the principal ideals of norm $n$, we shall use class group characters.  Let $C_K$ denote the ideal class group of $K$, and denote by $h_K$ its size which is 
the class number of $K$.  A class group character is a homomorphism $\psi: C_K \to \C^{\times}$.  We may think of such class group characters as maps 
$$ 
\psi: \{ \text{ non-zero ideals in } {\mathcal O}_K \} \to \C^{\times} 
$$ 
satisfying $\psi({\mathfrak a }{\mathfrak b}) = \psi({\mathfrak a}) \psi({\mathfrak b})$ and $\psi((\lambda)) = 1$ for every non-zero principal ideal $(\lambda)$.  We denote the dual group of class group characters by ${\widehat C}_K$.

If $\psi \in {\widehat C}_K$ is a class group character, then we define 
\begin{equation} 
\label{3.3} 
r(n,\psi) = r(n,\psi; D) = \sum_{N(\mathfrak a)=n} \psi({\mathfrak a}). 
\end{equation} 
Notice that ${\widehat C}_K$ always includes the principal character $\psi_0$ given by $\psi_0(\mathfrak a) =1$ for all ideals ${\mathfrak a}$.  In this case 
\begin{equation} 
\label{3.4} 
r(n,\psi_0) = \sum_{N(\mathfrak a)=n} 1 = (1*\chi_D)(n). 
\end{equation} 
The orthogonality relations for characters now allow us to express $R_D(n)$ in terms of $r(n,\psi)$: namely, 
\begin{equation} 
\label{3.5} 
R_D(n) = \frac{1}{h_K} \sum_{\psi \in {\widehat C}_K} r(n, \psi). 
\end{equation} 
With these preliminaries in place, we postpone a more detailed discussion of class group characters to Section \ref{sec8-class-group}.

\subsection{Representation by $x^2+dy^2$}  We now relate the concepts of the previous section to our specific problem of representing integers by the quadratic forms $x^2+dy^2$.  We will restrict attention to square-free integers $d$, which is sufficient for our purposes.  The problem of representing integers by $x^2+dy^2$ is naturally related to arithmetic in the field $K=\Q(\sqrt{-d}) = \Q(\sqrt{D})$ where $D$ denotes the fundamental discriminant 
\begin{equation} 
\label{3.6} 
D = \begin{cases} 
 -4d &\text{ if } d \equiv 1,2 \ \md{4}, \\ 
 -d & \text{ if } d \equiv 3\  \md{4}. 
 \end{cases}
 \end{equation}
Henceforth in the paper, we will adopt the following notational conventions. Unless explicitly stated otherwise, whenever we write $d$ we have in mind a square-free integer,  and 
corresponding to such $d$ will be the fundamental discriminant $D$ given in \eqref{3.6}, and the imaginary quadratic field $K = \Q(\sqrt{-d}) = \Q(\sqrt{D})$.
 Of course, $K$ and $D$ depend on $d$, but we will not indicate this explicitly. Sometimes we will additionally have a second positive square-free number $d'$, and $K', D'$ will be associated to it in the same way.

\begin{lemma}\label{lem3.1} Let $d \ge 1$ be square-free, and let $D, K$ be associated to $d$ as above. 
\begin{enumerate}
\item If $d \equiv 1$ or $2 \ \md{4}$, then the number of representations of $n$ by the quadratic form $x^2 + dy^2$ equals $2R_D(n)$, with the 
exception of the special case $d=1$ where it equals $4 R_{-4}(n)$.  
\item If $d\equiv 3\ \md{4}$ then the number of representations of $n$ by the quadratic form $x^2+dy^2$ is at most $2 R_D(n)$, with the exception of the special 
case $d=3$ where it is at most $6R_{-3}(n)$.   
\item If $d \equiv 7 \  \md{8}$ and $n$ is odd, then the number of representations of $n$ by the quadratic form $x^2+dy^2$ equals $2R_D(n)$.  
\end{enumerate}
\end{lemma}
\begin{proof}  If $d\equiv 1, 2 \ \md 4$, we have $D= -4d$, and the quadratic form $x^2 +dy^2$ is the principal form of discriminant $D$.  The result (1) now 
follows from our discussion in Section \ref{sec3.2}.  

If $d\equiv 3 \ \md 4$ then $D=-d$, and the principal form of discriminant $D$ is $x^2+xy + \frac{1+d}{4} y^2$.  The identity 
$$ 
x^2 + dy^2 = (x-y)^2 + (x-y)(2y) + \frac{1+d}{4} (2y)^2
$$ 
shows that the representations of $n$ as $x^2+dy^2$ are in bijective correspondence with the representations of $n$ as $X^2 + XY + \frac{1+d}{4} Y^2$ with $Y$ even.  
Since the total number of representations of $n$ as $X^2+XY +\frac{1+d}{4} Y^2$ (ignoring whether $Y$ is even or odd) equals $2R_D(n)$ (or $6R_{-3}(n)$ in the exceptional case $d=-3$), the upper bound stated in (2) follows. 

Finally, if $d\equiv 7 \ \md 8$ and $n$ is odd, then $\frac{1+d}{4}$ is even and so any representation of $n$ as $X^2+XY+ \frac{1+d}{4} Y^2$ must necessarily have $Y$ being even.  
Thus in this case the representations of $n$ by $X^2+XY + \frac{1+d}{4} Y^2$ equal the representations of $n$ by $x^2+dy^2$, and assertion (3) follows.  
\end{proof}

 \section{Proof of the upper bound} \label{sec4}

In this section we prove Theorem \ref{thm1.3}. It will follow from the following proposition. 
 
 \begin{proposition} \label{prop4.1}  Let $N$ be large, and $k$ an integer in the range \eqref{1.3}.   Let $d$ be a square-free integer with $d\leq \log N$.  
 Then the number of integers $n \in {\mathcal A}(N,k)$ that are represented by $x^2+dy^2$ is $\ll \frac{N}{2^k} (\log \log N)^3$.
  \end{proposition} 
 
 Before proving the proposition, let us deduce Theorem \ref{thm1.3}.   Note that if $d=d_1 d_2^2$ with $d_1$ square-free, then an integer represented by $x^2 +dy^2$ 
 is automatically represented by $x^2 +d_1 y^2$.  
 
 Using Proposition \ref{prop4.1}, it follows that the number of integers in ${\mathcal A}(N,k)$ that are represented by $x^2+dy^2$ for some $d$, $1\leq d\leq \Delta$ is 
$$ 
\ll  \Delta \frac{N}{2^k} (\log \log N)^3 \ll N (\log \log N)^{-1}, 
$$ 
since $\Delta \leq 2^k/k^4$.  This establishes Theorem \ref{thm1.3}. \vspace*{8pt}

To prove Proposition \ref{prop4.1} we require the following simple lemma. 
 \begin{lemma} \label{lem4.3}  Let $D$ be any fundamental discriminant apart from $D=1$.   For all $x\geq 1$ we have $\sum_{n\leq x} (1*\chi_D)(n) \ll x\log |D|$. 
  \end{lemma} 
  \begin{proof}  Suppose first that $x\leq |D|^2$.  Since $(1*\chi_D)(n) \leq \tau(n)$ (the number of divisors of $n$), the sum in question is 
  $\leq \sum_{n\leq x} \tau(n) \ll x \log (x+1) \ll x\log |D|$.  
  
  Now suppose that $x> |D|^2$, and note that 
  $$ 
  (1*\chi_D)(n) = \sum_{ab =n} \chi_D(b) = \sum_{\substack{ab =n \\ b \leq |D|}}\chi_D(b) + \sum_{\substack{ab=n \\ b>|D|}} \chi_D(b).
  $$ 
 Therefore 
 \begin{equation} 
 \label{4.1}  
 \sum_{n\leq x} (1*\chi_D)(n) = \sum_{b\leq |D|} \chi_D(b) \sum_{a\leq x/b} 1 + \sum_{a\leq x/|D|} \sum_{|D| < b\leq x/a} \chi_D(b). 
 \end{equation} 
  The first term on the right side of \eqref{4.1} contributes 
  $$ 
  \sum_{b\leq |D|} \chi_D(b) \Big( \frac{x}{b}+O(1)\Big) \ll \sum_{b\leq |D|} \Big( \frac{x}{b}+1\Big) \ll x \log |D|.
  $$ 
  Since $\chi_D$ is a non-principal character to the modulus $|D|$, it sums to zero over any interval of length $D$, and therefore $|\sum_{|D| < b\leq x/a} \chi_D(b)| \leq |D|$. It follows that the 
 second term on the right side of \eqref{4.1} contributes $\ll |D| \sum_{a\leq x/|D|} 1 \ll x$, and the lemma follows.
  \end{proof}

 \begin{proof}[Proof of Proposition \ref{prop4.1}] Let $d$ be square-free with $d \leq \log N$, and let $D$ be the fundamental discriminant associated to it (as given in \eqref{3.6}). Write 
 
 $\mathcal{R} = \mathcal{R}(d)$ for the set of all $r$ such that the primes dividing $r$ either divide $|D|$ or appear to exponent at least $2$ in the prime factorization of $r$. Suppose $n \in {\mathcal A}(N,k)$ is an integer that can be expressed as $x^2+dy^2$.  
 Write $n$ uniquely as $rs$, where $(r,s)=1$, $s$ is square-free and composed of primes not dividing $|D|$, and $r \in \mathcal{R}$. We have that $\Omega(r) \leq k$, and note that $\Omega(s) = k-\Omega(r)$.  
 
 By Lemma \ref{lem3.1} and \eqref{3.2} we know that if $n$ is representable by $x^2+dy^2$ then $(1*\chi_D)(n) \geq R_D(n) >0$.  Since $(1*\chi_D)$ is a non-negative multiplicative function, it follows that $(1*\chi_D)(s)> 0$, or equivalently that every prime $p|s$ satisfies $\chi_D(p)=1$ and therefore $(1*\chi_D)(s) = 2^{\Omega(s)}$.   
 Thus
 \begin{align*}
 \sum_{\substack{ n \in {\mathcal A}(N,k) \\ n =x^2 +dy^2}} 1  \leq \sum_{\substack{ rs \in {\mathcal A}(N,k)}} 2^{-\Omega(s)} (1*\chi_D)(s) 
 & = 2^{-k} \sum_{ rs \in {\mathcal A}(N,k)} 2^{\Omega(r)} (1*\chi_D)(s) \\ & \leq 2^{-k} \sum_{\substack{ r \in \mathcal{R} \\ \Omega(r) \leq k \\ r \leq N} } 2^{\Omega(r)} \sum_{s \leq N/r} (1*\chi_D)(s),
 \end{align*}
 where in the last step we used the nonnegativity of $1*\chi_D$ to take the sum over all $s \leq N/r$. Applying Lemma \ref{lem4.3} to the sum over $s$, we obtain $$ 
  \sum_{\substack{ n \in {\mathcal A}(N,k) \\ n =x^2 +dy^2}} 1 \ll \frac{N}{2^k} \log |D| \sum_{\substack{r \in \mathcal{R} \\ \Omega(r) \leq k} } \frac{2^{\Omega(r)}}{r}. 
  $$ 
  Now 
  \[
  \sum_{\substack{r \in \mathcal{R} \\ \Omega(r) \leq k}} \frac{2^{\Omega(r)}}{r}  \leq \prod_{p| |D|} \Big( 1 +  \sum_{j=1}^{k} \frac{2^j}{p^j} \Big) \prod_{p \nmid |D|} \Big(1 +\sum_{j=2}^{k} \frac{2^j}{p^j}\Big) \ll k \prod_{p| |D|} \Big(1 + \frac 2p\Big) \ll k (\log \log |D|)^2, 
  \]
  where the factor $k$ above arises from the prime $p=2$.  Since $k\ll \log \log N$ and $|D|\leq \log N$, the proposition follows. 
   \end{proof}

\section{Plan of the proof of the lower bound} \label{sec5}

We now turn to the proof of Theorem \ref{thm1.4}, which constitutes the bulk of the paper.  Let $N$ be large, recall that $k$ is an integer in the range \eqref{1.3}, and 
suppose in all that follows that
\[
k^3 2^k \leq \Delta \leq \log N. 
\]
We wish to bound the exceptional integers $n \in {\mathcal A}(N,k)$ that cannot be represented as $x^2+dy^2$ with $d$ below $\Delta$.  In fact, we shall consider only representations by such quadratic forms when $d$ is a prime lying in a suitable residue class, and show that most integers can be represented even with this further constraint.  

To state our results more precisely, we distinguish two cases according to whether the $2$-adic valuation $v_2(n)$ is $0$ or $1$ (or in other 
words whether $n\equiv 1 \ \md 2$ or $n \equiv 2 \ \md 4$).   Results for integers $n$ that are multiples of $4$ will be deduced easily from these 
cases.   Thus we define, for all $j =0$, $1$,
\[ 
{\mathcal A}_j(k) = \{ v_2(n) = j, \ \ \Omega(n)=k\}, 
\]
\[
{\mathcal A}_j(N,k) = \{ n \leq N, \ \ n \in {\mathcal A}_j(k)\}.
\]
Observe that 
\[ 
|\mathcal{A}_0(N, k)| = |\mathcal{A}(N, k)| - |\mathcal{A}(N/2, k - 1)|,
\] and
\[ 
|\mathcal{A}_1(N, k)| = |{\mathcal A}_0(N/2,k-1)| = |\mathcal{A}(N/2, k - 1)| - |\mathcal{A}(N/4, k - 2)|. 
\] 
Thus from Lemma \ref{lem1.2} we may deduce that for $k$ satisfying $|k-k_0| \leq \frac 13 k_0$ we have 
\begin{equation}
\label{ajuk} 
\mathcal{A}_j(N,k) = 2^{-j-1}\frac{N}{\log N} \frac{k_0^{k}}{k!} \Big(1 + O\Big(\frac{1+|k-k_0|}{k_0}\Big)\Big).
 \end{equation}
Here we recall that $k_0$ denotes $\log \log N$.

To each case $j = 0,1$ we associate a set $\mathcal{D}_j$ of primes. Below we let $W$ denote a parameter tending to 
infinity slowly with $N$; for definiteness, we set $W=\log \log \log N$.   With this choice of $W$, define
\begin{equation} 
\label{5.4} 
{\mathcal D}_0 = \Big \{ d\in \Big[\frac{\Delta}{\log \Delta}, \Delta\Big] \text{ prime},  \ \ d \equiv 7 \ \md 8, \ \ \chi_D(p)=1 \text { for } p\leq W \Big \}, 
\end{equation} 
\begin{equation} 
\label{5.5} 
{\mathcal D}_1 = \Big \{ d \in \Big[\frac{\Delta}{\log \Delta}, \Delta\Big ] \text{ prime},  \ \ d \equiv 1 \ \md 4, \ \ \chi_D(p)=1 \text { for \emph{odd} } p\leq W \Big \}.  
\end{equation} 
Here, as usual, $D$ denotes the fundamental discriminant associated to $d$ as given in \eqref{3.6}.  Thus $D=-d$ for $d\in {\mathcal D}_0$, and since 
$D\equiv 1 \ \md 8$ we have $\chi_D(2) =1$ automatically.  If $d$ is in ${\mathcal D}_1$ then $D= -4d$, and here $\chi_D(2)=0$.  
The primes in ${\mathcal D}_0$ lie in $\prod_{3\leq p \leq W} \frac{p-1}{2}$ reduced residue classes $\md {8\prod_{3\leq p\leq W} p}$, while those in 
${\mathcal D}_1$ lie in $\prod_{3\leq p \leq W} \frac{p-1}{2}$ reduced residue classes $\md {4\prod_{3\leq p\leq W} p}$.  Since $W$ is suitably 
small, a simple application of the prime number theorem in arithmetic progressions gives
\[
|{\mathcal D}_0|  = (1 + o(1)) \frac{1}{2^{\pi(W)+1} }\frac{\Delta}{\log \Delta}, \qquad |{\mathcal D}_1| = (1 + o(1)) \frac{1}{2^{\pi(W)}} \frac{\Delta}{\log \Delta}. 
\]
In particular, since $W = \log \log \log N$, and since \[ \log \Delta \gg k = (1 + o(1))\log \log N,\] we have the crude bounds
\begin{equation}\label{5.6b}
 |\mathcal{D}_0|, \  |\mathcal{D}_1| \gg \Delta (\log \Delta)^{-1 + o(1)}.
 \end{equation}

We are now ready to state our result on representing integers in ${\mathcal A}_j(N,k)$ using the binary 
quadratic forms $x^2 + dy^2$ with $d\in {\mathcal D}_j$.  From this result we shall swiftly deduce Theorem \ref{thm1.4}.

\begin{theorem} \label{thm5.1}  Suppose that $N$ is large, $k$ is an integer in the range 
\begin{equation}
\label{1.4range-wide} 
|k - k_0 | \leq 2k_0^{2/3},
\end{equation} 
where $k_0 := \log \log N$, and that $k^3 2^k \leq \Delta \leq \log N$.   
For $j = 0,1$ let ${\mathcal E}_j(N,k)$ denote the exceptional set of integers in ${\mathcal A}_j(N,k)$ that cannot be expressed as 
$x^2+dy^2$ for some $d\in {\mathcal D}_j$.  Then we have
$$
|{\mathcal E}_j(N,k)| \ll_{\eps}  \frac{|{\mathcal A}(N,k)|}{\log k_0}  + N k_0^{\eps - 5/6}.
$$ 
\end{theorem} 
\begin{proof}[Deducing Theorem \ref{thm1.4} from Theorem \ref{thm5.1}]   Extracting the largest power of $4$, we see that 
every $n \in {\mathcal A}(N,k)$ may be written uniquely as 
$n= 4^m r$ where $r$ is either in ${\mathcal A}_0(N/4^m, k-2m)$ or in ${\mathcal A}_1(N/4^m, k-2m)$.  Further, if $r$ can be represented as $x^2 +dy^2$ with 
$d\leq \Delta$, then plainly so can $n= 4^m r$.   Thus 
$$ 
|{\mathcal E}(N,k)| \leq \sum_{m\geq 0} \big( |{\mathcal E}_0(N/4^m, k-2m)|  + |{\mathcal E}_1(N/4^m, k-2m)| \big).
$$
First let us dispense with the terms $m \geq \log k_0$.  Bounding $|{\mathcal E}_0(N/4^m, k-2m)|  + |{\mathcal E}_1(N/4^m, k-2m)|$ trivially by $N/4^m$, we 
see that these terms contribute $\ll \sum_{m\geq \log k_0} N/4^m \ll N/k_0$, which is better than we need.  

For the terms with $m\leq \log k_0$, we wish to use Theorem \ref{thm5.1} to bound the quantity $|{\mathcal E}_j (N/4^m, k-2m)|$ (for $j=0,1$).   We must check that the required 
conditions there hold.   The condition on $\Delta$ is automatic: since $\Delta$ is assumed to be $\geq k^3 2^k$ it is clearly also $\geq (k-2m)^3 2^{k-2m}$.  The main condition to check is 
the analogue of \eqref{1.4range-wide} which here reads $|(k-2m) - \log \log (N/4^m) | \leq 2 (\log \log (N/4^m))^{2/3}$. To verify this, note that for $m\leq \log k_0$, one has $\log \log (N/4^m) = k_0 +O(1)$, and so the left side above is 
$\leq |k_0 - k| + 2m + O(1) \leq k_0^{2/3} + 2\log k_0 + O(1)$ since $k$ is in the range \eqref{1.3}.  Thus we may apply Theorem \ref{thm5.1}, and conclude 
that 
$$ 
|{\mathcal E}_0(N/4^m, k-2m)| +|{\mathcal E}_1(N/4^m, k-2m)| \ll_{\eps} \frac{|{\mathcal A}(N/4^m,k-2m)|}{\log k_0} + \frac{N}{4^m}  k_0^{\eps - 5/6}.
$$
Now applying Lemma \ref{lem1.2} we obtain 
\begin{align*}
|{\mathcal A}& (N/4^m, k-2m)|  \ll \frac{N}{4^m \log (N/4^m)} \frac{(\log \log (N/4^m))^{k-2m}}{(k-2m)!} 
\\ & \ll \frac{N}{4^m \log N} \frac{(\log \log N)^{k-2m}}{(k-2m)!}
\ll \frac{N}{4^m \log N} \frac{k_0^{k}}{k!} \Big(\frac{k}{k_0}\Big)^{2m} 
\ll \frac{|{\mathcal A}(N,k)|}{4^m},  
\end{align*}  
where the final estimate holds since $k/k_0 = 1 +O(k_0^{-1/3})$ and $m\leq \log k_0$ so that $(k/k_0)^{2m} \ll 1$.  We conclude that 
the contribution of the terms with $m \leq \log k_0$ may be bounded by 
$$ 
\ll \sum_{m\leq \log k_0} 4^{-m} \Big( \frac{|{\mathcal A}(N,k)|}{\log k_0} + N k^{\eps - 5/6}\Big) \ll 
\frac{|{\mathcal A}(N,k)|}{\log k_0} + N k_0^{-3/4}.   
$$ 

Combining this estimate with our bound for the larger range of $m$, we complete the deduction of Theorem \ref{thm1.4}.  
\end{proof}

Theorem \ref{thm5.1} will be deduced (in the next section) from the following four propositions which form the heart of our argument.    
Before stating these propositions, we introduce some notation that will be in place for the rest of our work.  We shall factorize $n$ as 
$\ns \nl$, where $\ns$ is composed only of primes below $W$, and $\nl$ is composed only of primes above $W$.  Further we define 
\begin{equation} 
\label{pw} 
\gamma_W = \prod_{p\leq W} \big(1- 1/p \big). 
\end{equation} 
For each choice of $j=0$ or $1$ we define 
\begin{equation} 
\label{Fj} 
F_j(n) = \frac{1}{|{\mathcal D}_j|} \sum_{d\in {\mathcal D}_j} \frac{|D|^{1/2}}{\pi \gamma_W} \frac{R_D(n)}{\tau(\ns)}. 
\end{equation} 
Note that if $n$ cannot be represented as $x^2 +dy^2$ with $d\in {\mathcal D}_j$, then $R_D(n)=0$ for all $d\in {\mathcal D}_j$ and 
therefore $F_j(n)=0$.   The proof is based on showing that for $n\in {\mathcal A}_j(k)$, the quantity $F_j(n)$ is usually close to its expected value of $1$, 
which is achieved by showing that $(F_j(n)-1)^2$ is small on average over $n$.  The four propositions below facilitate the calculation of this variance, which 
will be carried out in the next section.

\begin{proposition} \label{prop5.2}  Let $N$ be large and let $k$ be an integer in the range \eqref{1.4range-wide}.  The following statements hold for either choice of $j=0$ or $1$.  
Let $d$ be an element in ${\mathcal D}_j$, and let $D$ be the corresponding fundamental discriminant. Then 
\[
\frac{|D|^{1/2}}{\pi \gamma_W} \sum_{n \in {\mathcal A}_j(k)} \frac{R_D(n)}{\tau(\ns)} e^{-n/N} =  \sum_{n\in {\mathcal A}_j(k)} e^{-n/N}  + N O_{\eps}\big( k_0^{\eps - 5/6} + 
L(1,\chi_D)^{-1} k_0^{-2}\big),  
\] 
where $\gamma_W$ is as in \eqref{pw}.
\end{proposition}

Partial summation and \eqref{ajuk} easily allow us to give an asymptotic for the sum $\sum_{n\in {\mathcal A}_j(k)} e^{-n/N}$ appearing above.  
Write 
\[ 
\sum_{n\in {\mathcal A}_j(k)} e^{-n/N} = \int^{\infty}_0 e^{-u} |\mathcal{A}_j(uN, k)| du = 
\int_{1/\log N}^{\log N} e^{-u} |{\mathcal A}_j(uN,k)| du + O\big(\frac{N}{\log N}\big),
\] 
where we truncated the integral above using the trivial bound $|{\mathcal A}_j(uN,k)| \leq uN$ in the range $u \not \in [1/\log N, \log N]$.  Now using 
\eqref{ajuk} for $u \in [1/\log N, \log N]$ and the estimate 
 
 $\frac{k_0^k}{k!} \ll k_0^{-1/2} \log N$, which follows from \eqref{stirling-3}, we obtain that for $k$ in the range \eqref{1.4range-wide} 
\begin{equation}
\label{partial-sum} 
\sum_{n\in {\mathcal A}_j(k)} e^{-n/N} = 2^{-j - 1} \frac{N}{\log N} \frac{k_0 ^{k}}{k!} +  O( N k_0^{-5/6}).
\end{equation}

Note that there is a small subtlety in the application of \eqref{ajuk}, which is that $N$ must be replaced by $uN$ not only in the obvious term $\frac{N}{\log N}$, but also $k_0$ must be replaced by $\log \log (uN)$. We leave it to the reader to check that these changes have negligible effect for $u$ in the stated range.

Our next proposition considers averages of $R_D(n) R_{\tilde D}(n)$ for two different elements $d$, ${\widetilde d} \in {\mathcal D}_j$.  The answer 
will involve the character $\chi_{d {\tilde d}}$, which we now briefly introduce.  Since $d$ and ${\tilde d}$ are different primes that are congruent to each other $\md 4$, it follows 
that $d{\tilde d} \equiv 1 \ \md 4$ is a fundamental discriminant, and so the Kronecker symbol $\chi_{d {\tilde d}}$ is a primitive character to the modulus $d{\tilde d}$.  
This character is also closely connected to the product of characters $\chi_D \chi_{\tilde D}$.  Indeed in the case $j=0$ both characters are identical; and in the 
case $j=1$ the character $\chi_D \chi_{\tilde D}$ is the imprimitive character $\md {4d{\tilde d}}$ induced by the primitive character $\chi_{d {\tilde d}}$.

\begin{proposition} \label{prop5.3} Let $N$ be large and let $k$ be an integer in the range \eqref{1.4range-wide}. Let $j$ be $0$ or $1$. 
Let $d$ and ${\widetilde d}$ be two different elements in ${\mathcal D}_j$, and let $D$ and $\widetilde{D}$ denote the corresponding fundamental discriminants. If $d\equiv {\widetilde d} \ \md 8$ \textup{(}which is automatic when $j=0$\textup{)} then 
\begin{align} 
\label{5.10} 
 \frac{|D{\widetilde D}|^{1/2} }{\pi^2 \gamma_W^2} \sum_{n \in {\mathcal A}_j(k)} & \frac{R_D(n) R_{\tilde D}(n)}{\tau(\ns)^2} e^{-n/N} =  \big(2^j  +  O(W^{-1})\big)\gamma_W L(1,\chi_{d{\widetilde d}}) \sum_{n\in {\mathcal A}_j(k)} e^{-n/N} \nonumber\\ 
 &  + NO_{\eps}\Big( L(1,\chi_{d{\widetilde d}})k_0^{\eps - 5/6} +   L(1,\chi_D)^{-1} L(1, \chi_{\tilde D})^{-1} k_0^{\eps - 3}\Big),  
 \end{align} 
 while if $d \not\equiv {\widetilde d} \ \md 8$ \textup{(}which can only happen for $j=1$\textup{)} then 
\begin{equation} 
\label{5.11} 
 \sum_{n \in {\mathcal A}_j(k)} \frac{R_D(n) R_{\tilde D}(n)}{\tau(\ns)^2} e^{-n/N} =0.
 \end{equation}
 \end{proposition} 
 
 The next proposition concerns the case when $d = {\widetilde d}$, where an upper bound suffices.

 \begin{proposition} \label{prop5.4} Let $N$ be large and let $k$ be an integer in the range \eqref{1.4range-wide}.   Let $j=0$ or $1$, and let $d$ be an element of ${\mathcal D}_j$ with $D$ denoting the corresponding fundamental discriminant.  Then we have 
 \[
\frac{|D|}{\pi^2 \gamma_W^2} \sum_{n \in {\mathcal A}_j(k)} \frac{R_D(n)^2}{\tau(\ns)^2} e^{-n/N} \ll \frac{2^k N}{\gamma_W L(1,\chi_D)} \big( 
k_0^{-1/2} + L(1,\chi_D)^{-1}k_0^{-2}\big)    + |D|^{1/2} (\log |D|)^3 N. 
\]
\end{proposition} 

Finally, to complete our calculation of the average of $(F_j(n)-1)^2$, we shall need an asymptotic for the the average of $L(1,\chi_{d{\widetilde d}})$ 
appearing in Proposition \ref{prop5.3}.  

\begin{proposition} \label{prop5.5}  For each $j=0,1$ we have 
$$ 
\frac{1}{|\mathcal{D}_j|^{2}} \sum_{\substack{ d \neq \tilde{d} \in {\mathcal D}_j \\  d\equiv {\tilde d} \ \md 8} } L(1,\chi_{d\tilde{d}}) = \frac{1}{\gamma_W} \big(2^{-j}+ O(W^{-1})\big). 
$$ 
\end{proposition}

\section{Deducing Theorem \ref{thm5.1} from Propositions \ref{prop5.2}, \ref{prop5.3}, \ref{prop5.4} and \ref{prop5.5}}  \label{sec6}

We now deduce Theorem \ref{thm5.1} from the four propositions enunciated in the previous section.  Let $j$ be $0$ or $1$, and 
$k$ an integer in the range \eqref{1.4range-wide}.  Recall from \eqref{Fj} the definition of $F_j(n)$, and recall that $F_j(n) =0$ 
if $n$ cannot be represented as $x^2+dy^2$ with $d\in {\mathcal D}_j$.   Therefore, writing ${\mathcal E}_j(N,k)$ for the exceptional set as in Theorem \ref{thm5.1},
\begin{equation} 
\label{6.1} 
|{\mathcal E}_j(N,k)| \ll \sum_{n\in {\mathcal A}_j(k)} (F_j(n)-1)^2 e^{-n/N} = \sum_{n\in {\mathcal A}_j(k)} (F_j(n)^2 - 2 F_j(n) +1) e^{-n/N}. 
\end{equation}
We now invoke Propositions \ref{prop5.2}, \ref{prop5.3}, \ref{prop5.4}, and \ref{prop5.5} to bound the right side above.  To handle 
some error terms that arise, we require bounds for the average values of $L(1,\chi_D)^{-m}$ with $m=1$ and $2$. Although we can be more precise, it suffices to use \cite[Theorem 2]{GS} and \eqref{5.6b} to obtain 
\begin{equation} 
\label{6.2} 
\frac{1}{|\mathcal{D}_j|}\sum_{d\in {\mathcal D}_j} L(1,\chi_D)^{-m} \leq \frac{1}{|\mathcal{D}_j|}\sum_{\substack{d \leq \Delta \\ d\text{ odd} \\ \mu^2(d) = 1}} L(1,\chi_D)^{-m} \ll \frac{\Delta}{|\mathcal{D}_j|} \ll (\log \Delta)^{1 + o(1)}
\end{equation} for $j = 0,1$ and $m = 1,2$. 

A few further remarks on the application of \cite[Theorem 2]{GS} may be helpful. First, since we are dealing with moments where $m$ is bounded (albeit negative) we can exclude the contribution of exceptional characters, as remarked in the paragraph following the statement of \cite[Theorem 2]{GS}. Second, denoting by $X$ the random Euler product featuring in the statement of \cite[Theorem 2]{GS} then, as remarked in \cite[page 995]{GS}, $\mathbf{P}(L(1, X) \leq 1/t)$ decays doubly-exponentially as $t \rightarrow \infty$, so that the moments $\E L(1, X)^{-1}$ and $\E L(1,X)^{-2}$ are bounded.

 From Proposition \ref{prop5.2}, \eqref{6.2} (with $m = 1$) and the assumption that $\Delta \leq \log N$, it follows that 
\begin{equation} 
\label{6.3} 
 \sum_{n \in{\mathcal A}_j(k)} F_j(n) e^{-n/N}  = 
 \sum_{n\in {\mathcal A}_j(k)} e^{-n/N} + O_{\eps}(N k_0^{\eps - 5/6}).
 \end{equation} 
 
It remains to evaluate the terms involving $F_j(n)^2$ in \eqref{6.1}.   Expanding out the square we have 
 \begin{equation}\label{10-t}
 \sum_{n\in {\mathcal A}_j(k)} F_j(n)^2 e^{-n/N} = \frac{1}{\pi^2 \gamma_W^2 |{\mathcal D}_j|^2} \sum_{d, {\widetilde d} \in {\mathcal D}_j} |D \tilde D|^{1/2} \sum_{n\in {\mathcal A}_j(k)} \frac{R_D(n) R_{\widetilde D}(n)}{\tau(\ns)^2 } e^{-n/N}. 
 \end{equation}
Here we separate the diagonal terms $d={\widetilde d}$ from the off-diagonal terms $d\neq {\widetilde d}$.  By Proposition \ref{prop5.4} we see that the contribution of the diagonal terms is bounded by 
$$ 
\ll \frac{N}{|{\mathcal D}_j|^2} \sum_{d\in {\mathcal D}_j} \Big( 2^k \gamma_W^{-1} \big( L(1,\chi_D)^{-1}k_0^{-1/2} + L(1,\chi_D)^{-2} k_0^{-2}\big) + 
 |D|^{1/2} (\log |D|)^3\Big).
$$
Using \eqref{5.6b}, \eqref{6.2}, the Mertens bound $\gamma_W \geq 1/\log W = (\log \Delta)^{-o(1)}$, and that $\Delta \geq k^3 2^k$,  the above is 
\begin{equation} 
\label{6.4} 
\ll 2^k N k_0^{-1/2}(\log \Delta)^{2+o(1)}\Delta^{-1} + N (\log \Delta)^{5}\Delta^{-1/2} 
\ll N k_0^{-1}.
\end{equation} 
As for the off-diagonal terms in \eqref{10-t}, using Proposition \ref{prop5.3} we see that their contribution is 
\begin{align*}
\frac{1}{|{\mathcal D}_j|^2} \sum_{\substack{ d\neq {\widetilde d} \in {\mathcal D}_j \\ d\equiv {\widetilde d}\ \md 8}} \bigg(\big( &
2^j  + O(W^{-1})\big) L(1, \chi_{d\widetilde d})\gamma_W \sum_{n\in {\mathcal A}_j(k)} e^{-n/N} \\
&+ NO_{\eps}\big(L(1,\chi_{d{\widetilde d}})^{-1} k_0^{\eps - 5/6} + L(1,\chi_D)^{-1} L(1,\chi_{\widetilde D})^{-1} k_0^{\eps - 3 }\big)\Big).
\end{align*} 
Now using Proposition \ref{prop5.5}, \eqref{6.2}, and the bound

$\gamma_W^{-1} \ll (\log \Delta)^{o(1)}$, the above is 
\begin{equation} 
\label{6.5} 
\big(1 +O (W^{-1}) \big) \sum_{n\in {\mathcal A}_j(k)} e^{-n/N} + O_{\eps}( N k_0^{\eps - 5/6}). 
\end{equation} 

Combining \eqref{6.4} and \eqref{6.5} we conclude that 
 $$ 
 \sum_{n\in {\mathcal A}_j(k)} F_j(n)^2 e^{-n/N} = \big(1 +O(W^{-1}) \big) \sum_{n\in {\mathcal A}_j(k)} e^{-n/N} + O_{\eps}( N k_0^{\eps - 5/6}).
 $$
Taken together with \eqref{6.3}, it follows that 
\[
\sum_{n\in {\mathcal A}_j(k)} (F_j(n)-1)^2 e^{-n/N}  \ll_{\eps}
W^{-1}  \sum_{n\in {\mathcal A}_j(k)} e^{-n/N} +  N k_0^{\eps - 5/6} \ll  W^{-1} |{\mathcal A}(N,k)|+ N k_0^{\eps - 5/6},  \]
in view of \eqref{partial-sum} and Lemma \ref{lem1.2}.  
Using this estimate in \eqref{6.1}, and recalling that $W = \log \log \log N = \log k_0$, Theorem \ref{thm5.1} follows. 

\section{Proof of Proposition \ref{prop5.5}} \label{sec7}

In the proof below it is convenient to set 
\[
K = (\log \Delta)^{20}, \qquad M = \Delta^{3/2}. 
\] 
Suppose $d$ and ${\widetilde d}$ are distinct elements in ${\mathcal D}_j$ with $d \equiv {\widetilde d} \ \md 8$.  Then 
$d{\widetilde d}$ is a square-free integer $\equiv 1 \ \md 8$, and is thus a fundamental discriminant.  
Since $d{\tilde d}\leq \Delta^2$, partial summation and  the P\'olya-Vinogradov inequality give
\begin{equation} 
\label{7.1} 
L(1,\chi_{d{\tilde d}}) = \sum_{n\leq M} \frac{\chi_{d{\tilde d}}(n)}{n} + \int_{M}^{\infty} \sum_{M< n\leq t} \chi_{d\tilde d}(n) \frac{dt}{t^2} 
= \sum_{n\leq M} \frac{\chi_{d{\tilde d}}(n)}{n} +  O(\Delta^{-1/4}). 
\end{equation}

We first show that (when summed over $d$ and ${\tilde d}$) the terms with $n > K$ contribute a negligible amount.  
Here we extend the sum over $d {\tilde d}$ to all discriminants below $\Delta^2$ that are $1 \ \md 8$.  Recall that a 
discriminant is an integer $\ell \equiv 0$ or $1 \ \md 4$, and that every discriminant $\ell$ may be written uniquely as 
$\ell_0 r^2$ where $\ell_0$ is a fundamental discriminant.  For every discriminant $\ell$ we may define the Kronecker symbol $\chi_\ell$ 
exactly as in Subsection 3.1, and it defines a quadratic character $\md \ell$, possibly imprimitive and induced from the primitive character 
$\chi_{\ell_0}$.  Thus,  using Cauchy-Schwarz, we find  
\[
\sum_{\substack{ d, \widetilde{d} \in {\mathcal D}_j \\ d\neq {\widetilde d}  \\ d\equiv {\widetilde d} \ \md 8}} \Big| \sum_{ K \leq n \leq M} \frac{\chi_{d{\tilde d}}(n)}{n}\Big| 
\leq \sum_{\substack{ d\leq \Delta^2 \\ d\equiv 1 \ \md 8}} \Big|  \sum_{ K \leq n \leq M} \frac{\chi_{d}(n)}{n}\Big|  
\leq \Delta \Big(  \sum_{\substack{ d\leq \Delta^2 \\ d\equiv 1 \ \md 8}} \Big|  \sum_{K \leq n \leq M} \frac{\chi_{d}(n)}{n}\Big|^2 \Big)^{1/2}. 
\]
Expanding the square, we obtain 
\begin{equation} 
\label{7.2} 
\ \sum_{\substack{ d\leq \Delta^2 \\ d\equiv 1 \ \md 8}} \Big|  \sum_{ K \leq n \leq M} \frac{\chi_{d}(n)}{n}\Big|^2  
= \sum_{K \leq n_1, n_2 \leq M } \frac{1}{n_1 n_2} \sum_{\substack{ d\leq \Delta^2 \\ d\equiv 1 \ \md 8}} \chi_d(n_1n_2). 
\end{equation} 

Write $n_1 n_2$ as $2^a n$ where $n$ is odd.  Since $d\equiv 1 \ \md 8$, $\chi_d(2)=1$, and therefore $\chi_d(n_1 n_2) = 
\chi_d(n)$ may also be expressed as the Jacobi symbol $(\frac{d}{n})$.   Now the Jacobi symbol $( \frac{\cdot}{n})$ is a quadratic 
character $\md n$, and is non-principal exactly when $n$ is not a square; or, in other words, when $n_1 n_2$ is neither a square 
nor twice a square.  Thus, when $n_1 n_2$ is neither a square nor twice a square we find by the P{\' o}lya-Vinogradov inequality 
$$ 
\sum_{\substack{ d\leq \Delta^2 \\ d\equiv 1 \ \md 8}} \Big( \frac{d}{n}\Big) 
= \sum_{8k+1 \leq \Delta^2} \Big(\frac{8k+1}{n}\Big) = \Big( \frac{8}{n}\Big) \sum_{k \leq (\Delta^2-1)/8} \Big(\frac{k+\overline{8}}{n}\Big) 
\ll \sqrt{n} \log n, 
$$ 
where $\overline 8$ denotes the inverse of $8$ modulo $n$.  If $n_1 n_2$ is a square or twice a square, then the inner sum over $d$ in \eqref{7.2} 
is clearly $O(\Delta^2)$.  Thus we obtain that the quantity in \eqref{7.2} is 
$$ 
\ll \Delta^2 \sum_{\substack{ K \leq n_1, n_2 \leq M \\ n_1 n_2 = \square,  2 \square}} \frac{1}{n_1 n_2}  + \sum_{K \leq n_1, n_2 \leq M } 
\frac{\sqrt{n_1 n_2} \log (n_1n_2)}{n_1  n_2}.
$$ 
The second term above is easily bounded by $\ll M\log M$.   Now consider the first term, where we handle the case $n_1 n_2 =m^2$ with the case $n_1 n_2 =2m^2$ treated 
in the same manner.  The terms $n_1 n_2 =m^2$ contribute, with $\tau(\cdot)$ denoting the divisor function  
$$ 
\leq \Delta^2 \sum_{K \leq m \leq M} \frac{\tau(m^2)}{m^2} \leq \frac{\Delta^2}{K} \sum_{m\leq M} \frac{\tau(m^2)}{m} \ll \frac{\Delta^2}{K} \prod_{p\leq M}\Big(\sum_{j=0}^{\infty} \frac{\tau(p^{2j})}{p^j} \Big) 
\ll \frac{\Delta^2}{K} (\log M)^3.
$$ 
We conclude that the quantity in \eqref{7.2} is 
$$ 
\ll \frac{\Delta^2}{K} (\log \Delta)^3 + M \log M \ll  \Delta^2 (\log \Delta)^{-10}. 
$$ 

Combining the above argument with \eqref{7.1} we find that 
\begin{equation} 
\label{7.4} 
\sum_{\substack{ d \neq \tilde{d} \in {\mathcal D}_j  \\ d\equiv {\tilde d} \ \md 8} } L(1,\chi_{d{\tilde d}}) = \sum_{\substack{ d \neq \tilde{d} \in {\mathcal D}_j  \\ d\equiv {\tilde d} \ \md 8} }  
\sum_{n\le K } \frac{\chi_{d {\tilde d}}(n)}{n}  +O\big( \Delta^2 (\log \Delta)^{-5} \big). 
\end{equation}  
To analyse the main term above, write $n\leq K$ uniquely as $n= frm^2$ where $f$ and $r$ are square-free with all prime factors of $f$ being below $W$ and 
all prime factors of $r$ being above $W$ (in particular, $r$ is odd, and note that $r$ could be $1$).  Note that for all $p$, $3\leq p\leq W$, we have $\chi_{d{\widetilde d}}(p) = \chi_{D}(p) \chi_{\widetilde D}(p) =1$.  
Since $d \equiv {\widetilde d} \ \md 8$, we have $d{\widetilde d} \equiv 1 \ \md 8$, and it follows also that $\chi_{d{\widetilde d}} (2) =1$.  Finally since $d$ and $\widetilde d$ 
are primes in the range $[\Delta/\log \Delta, \Delta]$, and $m^2 \leq n \leq K = (\log \Delta)^{20}$ we know that $(d{\widetilde d},m^2)=1$ and therefore $\chi_{d{\widetilde d}}(m^2) =1$.  
Thus $\chi_{d{\widetilde d}}(n)$ equals the Jacobi symbol $(\frac{d{\widetilde d}}{r})$, which for given $r$ is a quadratic character that is principal when $r=1$ and non-principal for 
$r>1$.   With this notation, the main term in \eqref{7.4} may be expressed as 
\begin{equation} 
\label{7.5} 
\sum_{n=fr m^2 \leq K} \frac{1}{n} \sum_{\substack{ d \neq \tilde{d} \in {\mathcal D}_j  \\ d\equiv {\tilde d} \ \md 8} } \Big( \frac{d{\widetilde d}}{r} \Big). 
\end{equation} 

We now show that the asymptotic in Proposition \ref{prop5.5} arises from the contribution of $r=1$ here, while the terms with $r >1$ contribute a negligible amount.  When $r=1$, note 
that $(\frac{d \widetilde d}{r}) =1$.  Since $d$ and ${\widetilde d}$ range over primes in $[\Delta/\log \Delta, \Delta]$ in suitable progressions modulo $8\prod_{3\leq p\leq W} p$,
and this modulus is $\leq e^{(1+o(1))W} = (\log \Delta)^{1+o(1)}$,  by the prime number theorem in arithmetic progressions it follows that 
$$ 
\sum_{\substack{ d \neq \tilde{d} \in {\mathcal D}_j  \\ d\equiv {\tilde d} \ \md 8} }1 = 2^{-j} |{\mathcal D}_j|^2  +O(\Delta^2 (\log \Delta)^{-10}). 
$$ 
When $j=0$ the condition $d \equiv {\widetilde d} \ \md 8$ is automatic, while when $j=1$ we only know from the definition that $d \equiv {\widetilde d} \ \md 4$ and the 
extra constraint $\md 8$ accounts for the factor $2^j=2$ above.  Now the unrestricted sum over $n$ satisfies 
$$ 
\sum_{n=fm^2 \ge 1}  \frac{1}{n} =\prod_{p\leq W} \big(1- p^{-1}\big)^{-1} \prod_{p>W} \big( 1- p^{-2}\big)^{-1} = \gamma_W^{-1} \big( 1+ O(W^{-1})\big), 
$$ 
while the tail  $\sum_{n= fm^2 >K} 1/n$ may be bounded by 
$$ 
\sum_{f|\prod_{p\leq W} p} \frac 1f \sum_{m \geq \sqrt{K/f}} \frac 1{m^2} \ll \frac{1}{\sqrt{K}} \sum_{f|\prod_{p\leq W} p} \frac{1}{\sqrt{f}} \ll \frac{\log \log N}{\sqrt{K}} \ll (\log \Delta)^{-9}. 
$$ 
We conclude that the terms with $r=1$ in \eqref{7.5} contribute 
$$ 
\Big( 2^{-j} |{\mathcal D}_j|^2  +O(\Delta^2(\log \Delta)^{-10})\Big) \Big( \gamma_W^{-1} \big( 1+ O(W^{-1}\big)+ O((\log \Delta)^{-9})\Big) 
$$
This is $2^{-j} |\mathcal{D}_j|^2   \gamma_W^{-1} ( 1+ O(W^{-1}))$, matching the expression in the proposition. 

It remains to show that the contribution to \eqref{7.5} of terms with $r>1$ is negligible.  Given $d\in {\mathcal D}_j$, consider the sum over ${\widetilde d}$ in \eqref{7.5}, which is 
$$ 
\sum_{\substack {\widetilde d \in {\mathcal D}_j \\ {\widetilde d}\neq d \\ {\widetilde d}\equiv d \ \md 8}} \Big( \frac{d{\widetilde d}}{r} \Big) 
= \Big( \frac{d}{r} \Big) \sum_{\substack {\widetilde d \in {\mathcal D}_j \\ {\widetilde d}\equiv d \ \md 8}} \Big( \frac{{\widetilde d}}{r} \Big)  + O(1)
=\Big( \frac{d}{r} \Big)  \sum_{ a\ \md r} \Big( \frac{a}{r} \Big) \sum_{\substack {\widetilde d \in {\mathcal D}_j \\ {\widetilde d}\equiv d \ \md 8\\ d\equiv a  \ \md r}} \!\!\!\! 1 +O(1). 
$$ 
Now the sum over ${\widetilde d}$ above counts primes in $[\Delta/\log \Delta,\Delta]$ lying in a suitable number of arithmetic progressions modulo
$8r \prod_{3\leq p\leq W} p$.  Since the modulus is $\ll K e^{(1+o(1))W} \leq (\log \Delta)^{22}$, an application of the prime number theorem in arithmetic 
progressions shows that the above equals
$$ 
\Big( \frac{d}{r} \Big)  \sum_{ a\ \md r} \Big( \frac{a}{r} \Big) \Big( \frac{1}{\phi(r)} \frac{|{\mathcal D}_j|}{2^j}  + O\big(\Delta(\log \Delta)^{-40}\big) \Big) 
= O\big( \Delta (\log \Delta)^{-20}\big), 
$$ 
upon noting that the main terms cancel (since $(\frac{\cdot}{r})$ is a non-principal character) and that $r\leq K = (\log \Delta)^{20}$.  Thus the contribution of the terms $r>1$ to \eqref{7.5} 
is 
$$ 
\ll \sum_{n \leq K} \frac{1}{n} |{\mathcal D}_j| \Delta (\log \Delta)^{-20} \ll \Delta^2(\log \Delta)^{-19}.
$$ 

Combining this with our evaluation of the terms with $r=1$, we conclude that the quantity in \eqref{7.5} is $2^{-j} \gamma_W^{-1} |{\mathcal D}_j|^2 (1+O(W^{-1}))$, and using this in 
\eqref{7.4} the proof of Proposition \ref{prop5.5} is complete.

\section{Class group  \texorpdfstring{$L$}{}-functions} \label{sec8-class-group}

We begin by recalling properties of class group $L$-functions over general number fields.   In our work we will only need the special cases of quadratic and biquadratic extensions.  
Let $K$ be a number field of degree $m$ and discriminant $D_K$.  Let $\Psi$ be a character of the class group of $K$, and let $L(s,\Psi)$ denote the corresponding $L$-function.  Recall that $L(s,\Psi)$ is defined by 
\begin{equation} 
\label{8.1} 
L(s, \Psi) = \sum_{\mathfrak a \neq 0} \Psi(\mathfrak a) N(\mathfrak a)^{-s} = \prod_{\mathfrak p} \big(1 - \Psi(\mathfrak p) N(\mathfrak p)^{-s} \big)^{-1}, 
\end{equation} 
where both the Dirichlet series and Euler product above converge absolutely in the half plane $\sigma > 1$.  In the half-plane $\sigma >1$, we define a 
holomorphic branch of $\log L(s,\Psi)$ by setting 
\begin{equation} 
\label{8.2} 
\log L(s,\Psi) = \sum_{\mathfrak p} \log \big( 1- \Psi(\mathfrak p) N(\mathfrak p)^{-s}\big)^{-1} = 
\sum_{\mathfrak p} \sum_{j=1}^{\infty} \frac{1}{j} \Psi(\mathfrak p)^j N(\mathfrak p)^{-js}.  
\end{equation} 
The Dirichlet series coefficients of $L(s,\Psi)$ are bounded in absolute value by the corresponding coefficients of the Dedekind zeta function $\zeta_K(s)$, which in turn 
are no more than the coefficients of $\zeta(s)^m$ (which has coefficients given by the $m$-divisor function).  Further, the coefficients of $\log L(s,\Psi)$ (as 
defined above) are supported on prime powers, and bounded in size by the coefficients of $\log \zeta_K(s)$ (defined as above for the principal character $\Psi_0$), and thus
 are no more than $m/j$ on the prime powers $p^j$.  In particular, we note that in the half-plane $\sigma >1$ 
 \begin{equation} 
 \label{8.3} 
|\log L(s,\Psi)| \leq m \log \zeta(\sigma) \leq m \log \Big( \frac{\sigma}{\sigma-1}\Big),
\end{equation} 
with the second bound being a standard bound for $\zeta$ (see for instance \cite[Corollary 1.4]{mv}).

We now collect together some classical bounds for $L(s,\Psi)$, along with describing a zero free region for $L(s,\Psi)$ and bounds for $|\log L(s,\Psi)|$ inside 
the zero free region.

\begin{lemma} \label{lem8.1}   Let $K$, $\Psi$, and $L(s,\Psi)$ be as above.  Then the following statements hold.
\begin{enumerate}\item  Suppose that $\Psi$ is not the principal character. Then $L(s,\Psi)$ extends to an entire function, and uniformly in the region $\sigma \geq 0$ satisfies the bound  
\begin{equation} 
\label{8.4} 
|L(\sigma+it,\Psi)| \ll_m \Big( (|D_K| (1+|t|)^m)^{(1-\sigma)/2} + 1 \Big) (\log (|D_K|(1+|t|)))^{m}. 
\end{equation}
For every $\eps >0$ there is a constant $C=C(m,\eps) >0$ such that the region 
\[
{\mathcal R}_0 = {\mathcal R}_0(\eps) = \{ \sigma \geq 1- C |D_K|^{-\eps}, \ \ |t| \leq |D_K| \} 
\]
is free of zeros of $L(s,\Psi)$. Thus $\log L(s,\Psi)$ extends analytically to the region ${\mathcal R}_0$, and moreover in the sub-region 
\[
{\mathcal R} = {\mathcal R}(\eps) =  \Big \{ \sigma \geq 1- \tfrac{1}{2} C |D_K|^{-\eps}, \ \ |t| \leq \tfrac{1}{2} |D_K| \Big\}
\]
we have the bound
\begin{equation}\label{8.5}
|\log L(s,\Psi)| \leq 6m\eps \log |D_K| + O_{m,\eps}(1). 
\end{equation}

\item Suppose that $\Psi$ is the principal character so that $L(s,\Psi)$ is the Dedekind zeta-function $\zeta_K(s)$ of the field $K$.  The 
Dedekind zeta function extends to a meromorphic function, with a single simple pole at $s=1$.    
The convexity bound \eqref{8.4} holds provided $|t| \geq 1$, while for $|t|\leq 1$ the same bound holds for $|(s-1)\zeta_K(s)|$. 
The region ${\mathcal R}_0$ is free of zeros of $\zeta_K(s)$, and the function $\log ((s-1)\zeta_K(s))$ extends analytically to the 
region ${\mathcal R}_0$.   The bound \eqref{8.5} holds for $|\log \zeta(s)|$ in the subregion ${\mathcal R}$ provided $|s-1| \geq 1$, 
and for points in ${\mathcal R}$ with $|s-1| \leq 1$ the same bound holds for $|\log ((s-1) \zeta_K(s))|$ instead. 
\end{enumerate}
\end{lemma} 
\begin{proof}    Suppose first that $\Psi$ is nonprincipal.  
The analytic continuation of $L(s,\Psi)$ to the entire plane is due to Hecke (for a modern account, see for example Chapter 7 of \cite{Nark}).

The bound in \eqref{8.4} is a standard convexity bound, and for instance may be obtained 
from Lemma 4 of Fogels \cite{Fogels1}.    Fogels's paper \cite{Fogels1} established a classical zero-free region for $L(s,\Psi)$ of the form $\sigma \geq 1- c/\log (|D|(1+|t|))$ for a suitable constant $c>0$, when the character $\Psi$ is complex.   In the case of a real character $\Psi$, the same region is free of zeros of $L(s,\Psi)$ except for the possibility of a simple zero at $1-\delta$ for a real number $\delta$ (analogous to the Siegel zero for Dirichlet $L$-functions).  Analogously to the Brauer--Siegel theorem, Fogels \cite{Fogels2} shows, by reducing to 
Brauer's work, that $\delta \geq C(m, \eps) |D_K|^{-\eps}$.  Thus the region ${\mathcal R}_0$ is free of zeros of $L(s,\Psi)$.  

The bound \eqref{8.5} on $|\log L(s,\Psi)|$ in the narrower region ${\mathcal R}$ follows by an application of the Borel--Caratheodory lemma using the preliminary bounds 
\eqref{8.3} and \eqref{8.4}, as we shall now see.  Let $z_0 = 1+ \frac C2 |D_K|^{-\eps} +i t$ with $|t|\leq |D_K|/2$, and put $r= C|D_K|^{-\eps}$ and $R = \frac 32 C |D_K|^{-\eps}$.   The function $f(z) = \log L(z,\Psi)$ is holomorphic inside the circle of radius $R$ centered at $z_0$ (since this is contained in the region ${\mathcal R}_0$), and for $z$ inside this larger circle it satisfies the bound 
\[
\Re f(z) = \log |L(z,\Psi)| \leq m \log \log |D_K| + O_{m,\eps}(1),   
\]
 
since by \eqref{8.4} we have $|L(z,\Psi)| \ll_{m,\epsilon} (\log |D_K|)^m$.  Further, by \eqref{8.3} 
\[
|f(z_0)| = | \log L(1+\tfrac{1}{2} C |D_K|^{-\eps} + it, \Psi)| \leq m \eps \log |D_K| + O_{m,\eps}(1). 
\]
The Borel-Carath{\' e}odory lemma (see for example Section 5.5 of \cite{titchmarsh}) 
now shows that for $z$ inside the smaller circle $|z-z_0| \leq r$ one has 
\begin{align*}
|f(z)| &\leq \frac{2r}{R-r} \sup_{|z-z_0| \leq R}\Re f(z) + \frac{R+r}{R-r} |f(z_0)| \\
&\leq 4 m \log \log |D_K| + 5m \eps \log |D_K| + O_{m,\eps}(1) \le 6m \eps \log |D_K| + O_{m,\eps}(1). 
\end{align*} 
This establishes \eqref{8.5} for all $s= \sigma+it$ with $|t|\leq \frac{1}{2}|D_K|$ and $1-\frac{1}{2} C |D_K|^{\eps} \leq \sigma \leq 1+ \frac{3}{2}C |D_K|^{-\eps}$.  
When $\sigma > 1+\frac {3}{2}C |D_K|^{-\eps}$ (and $|t| \leq \frac{1}{2}|D_K|$) the bound in \eqref{8.5} follows at once from \eqref{8.3}, and this completes the proof in the 
case of nonprincipal $\Psi$.

The case when $\Psi$ is principal follows in the same way.  The only difference is that the Dedekind zeta function has a pole at $s=1$, so that near $1$ we 
deal with $(s-1) \zeta_K(s)$ instead.  
\end{proof}

To prove Propositions \ref{prop5.2}, \ref{prop5.3} and \ref{prop5.4} we shall make use of the expression \eqref{3.5} of $R_D(n)$ in terms of the coefficients of the class group $L$-functions $r(n,\psi)$.   As consequences of  Lemma \ref{lem8.1}, we now show that in such expressions the contribution of most class group characters $\psi$ is negligible.  
The main lemmas we will prove in this section are Lemmas \ref{lem8.2}, \ref{lem8.4} and \ref{lem8.5}. The analytic details are very similar across all three, so we will only give complete details in the proof of Lemma \ref{lem8.2}.

Recall the convention introduced in Section \ref{sec5}, namely that for integer $n$ we write $n = \nl \ns$, where $\nl$ has only prime factors $\leq W$, and $\ns$ only prime factors $>W$.

\begin{lemma}\label{lem8.2}   Let $N$ be large and $k$ be an integer in the range \eqref{1.4range-wide}.  Let $j$ be $0$ or $1$, and let $d$ be an element of ${\mathcal D}_j$ with $D$ denoting the corresponding fundamental discriminant.  Let $\psi$ be a non-principal class group character of the quadratic field $K=\Q(\sqrt{D})$.  Then 
\[
\sum_{n\in {\mathcal A}_j(k)} \frac{r(n,\psi)}{\tau(\ns)} e^{-n/N} \ll N(\log N)^{-100}.
\]
\end{lemma} 
\begin{proof} 
The key idea here and in the proofs of Lemmas \ref{lem8.4} and \ref{lem8.5} is to follow Selberg \cite{selberg-54} and introduce, for any $z \in \C$ with $|z| = 1$, the Dirichlet series
\[
\mathcal{F}(s;z,j) := \sum_{v_2(n) = j} \frac{r(n,\psi)}{\tau(\ns)} z^{\Omega(n)} n^{-s}.  
\]
Later, we will recover the condition $\Omega(n) = k$ (which defines the set $\mathcal{A}_j(k)$) by Fourier inversion.

Since (by \eqref{3.3}, \eqref{3.4}) $|r(n,\psi)| \leq \tau(n)$ we see that ${\mathcal F}(s,z,j)$ converges absolutely in the half plane Re$(s) =\sigma >1$, and further satisfies the bound 
\begin{equation} 
\label{8.7} 
|{\mathcal F}(s; z,j)| \leq \sum_{n=1}^{\infty} \tau(n) n^{-\sigma} = \zeta(\sigma)^2 \leq \Big( \frac{\sigma}{\sigma-1}\Big)^2,
\end{equation} 
using \cite[Corollary 1.4]{mv} in the last step.
Further by Mellin inversion we have, 
 setting $c= 1+1/\log N$,
\begin{equation} \label{8.7a}
\sum_{n\in {\mathcal A}_j(k)} \frac{r(n,\psi)}{\tau(\ns)}z^{\Omega(n)} e^{-n/N} = \frac{1}{2\pi i} \int_{c-i\infty}^{c+i\infty} {\mathcal F}(s; z,j) N^s \Gamma(s) ds.
\end{equation} 
By Stirling's formula $|\Gamma(\sigma+ it)| \ll (1+|t|)^{\sigma-1/2} e^{-\pi |t|/2}$ uniformly for $\sigma$ in bounded intervals (see, for instance, (C.19) of 
\cite{mv}).  Combining this with the bound \eqref{8.7} we find that the tails of the integral in \eqref{8.7a} above where $|\text{Im}(s)| \geq (\log \log N)^2$ contribute 
\[
\ll \int_{|t| > (\log \log N)^2} N^c (\log N)^2 (1+|t|)^{c-1/2} e^{-\pi |t|/2} dt \ll N (\log N)^{-100}. 
\]
Thus, writing $T= (\log \log N)^2$, 
\begin{equation} 
\label{8.8}  
\sum_{n\in {\mathcal A}_j(k)} \frac{r(n,\psi)}{\tau(\ns)} z^{\Omega(n)}e^{-n/N} = \frac{1}{2\pi i} \int_{c-iT}^{c+iT} {\mathcal F}(s; z,j) N^s \Gamma(s) ds + O\big( N (\log N)^{-100}\big).  
\end{equation} 

To estimate the truncated integral here, we shall extend ${\mathcal F}(s;z,j)$ analytically a little to the left of the $1$-line and shift contours.  
To extend ${\mathcal F}(s;z,j)$ analytically we shall compare it with $L(s,\psi)^z$.  Note that when Re$(s) >1$ we may define $L(s,\psi)^z$ by 
the Euler product $\prod_{\mathfrak p} (1- \psi(\mathfrak p)/N(\mathfrak p)^s)^{-z}$, and this product converges absolutely when Re$(s)>1$.  Further 
we may extend $L(s,\psi)^z$ analytically to a wider region by writing it as $\exp(z\log L(s,\psi))$ and using the analytic continuation described in Lemma \ref{lem8.1}.   
Thus, define 
\[
{\mathcal G}(s;z,j) =  {\mathcal F}(s;z,j) L(s, \psi)^{-z}, 
\]
which is, to start with, analytic in the half-plane $\sigma >1$.  The definition of ${\mathcal F}(s;z,j)$ permits us (in this region) to write ${\mathcal G}(s;z,j)$ as an Euler 
product $\prod_p {\mathcal G}_p(s;z,j)$, whose factors we now describe.  For $p>W$ we have 
\begin{align*}
{\mathcal G}_p(s;z,j) &= \Big( \sum_{j=0}^{\infty} z^j r(p^j,\psi) p^{-js} \Big) \prod_{ {\mathfrak p}|p} \big(1 -\psi(\mathfrak p)N(\mathfrak p)^{-s} \big)^z 
\\
&= \Big( 1+  zr(p,\psi) p^{-s} + O(p^{-2\sigma}) \Big) \big( 1 - z p^{-s} \sum_{N(\mathfrak p) =p} \psi(\mathfrak p) + O( p^{-2\sigma})\big) 
\\
&=  1+ O(p^{-2\sigma}). 
\end{align*} 
 For $p$ with $3\leq p\leq W$ we have 
\[
 {\mathcal G}_p(s;z,j) = \Big( \sum_{j=0}^{\infty} \frac{r(p^j,\psi)}{(j+1)} z^j p^{-js}\Big) \prod_{ {\mathfrak p}|p} \big(1 -\psi(\mathfrak p)N(\mathfrak p)^{-s} \big)^z  
 = 1+ O(p^{-\sigma}).  
 \]
 Finally for $p=2$ we have 
 \[
 {\mathcal G}_2(s;z,j) = \left\{\begin{array}{ll} 
 z 2^{-s-1} r(2,\psi) \prod_{\mathfrak p| 2} \big(1- \psi(\mathfrak p)N(\mathfrak p)^{-s}\big)^{z} &\text{ if  } j=1 \\ 
 \prod_{\mathfrak p| 2} \big(1- \psi(\mathfrak p) N(\mathfrak p)^{-s}\big)^{z}  &\text{ if } j =0,
 \end{array}\right. 
 \]
 and in both cases this is $1+O(2^{-\sigma})$.  From these remarks we see that the Euler product ${\mathcal G}_p(s;z,j)$ 
 converges absolutely in the region $\Re(s)>\frac{1}{2}$, and defines a holomorphic function of $s$ in that region.  Moreover, in the region 
 $\sigma \geq \frac{3}{4}$, we have the bound 
 \begin{equation} 
 \label{8.10} 
 |{\mathcal G}(s;z,j)| \ll \prod_{p\leq W} \big( 1 + O(p^{-3/4})\big) \ll \exp( W^{1/4}). 
 \end{equation} 
 
For the rest of the paper, we fix the domain
\begin{equation}\label{w-def-first}  \mathcal{W} := \{s \in \C:  1 - 2(\log N)^{-1/2} < \Re s < 2,\; |\Im s| < 2(\log \log N)^2\} .\end{equation} 
 
Applying Lemma \ref{lem8.1} 
with $\eps =\frac{1}{100}$ (and $m=2$), we see that (keeping in mind $(\log N)^{1/2} \leq |D| \leq \log N$) the function  $\log L(s,\psi)$ is analytic in $\mathcal{W}$, and satisfies $|\log L(s,\psi)| \leq \frac 18 \log |D| +O(1)$ 
here.  Therefore 
\[
\mathcal{F}(s;z,j)  = \exp(z \log L(s,\psi)) {\mathcal G}(s;z,j) 
\]
is also analytic in $\mathcal{W}$ and by \eqref{8.10} satisfies in this region 
\[
|{\mathcal F}(s;z,j)| \ll \exp\Big( \tfrac{1}{8} \log |D| + W^{1/4} \Big) \ll \log N. 
\]
We now return to the integral in \eqref{8.8}, and replace the line of integration from $c-iT$ to $c+iT$ by integrals along the following three line segments: 
(i) the horizontal line segment from $c-iT$ to $1-(\log N)^{-1/2} -iT$, (ii) the vertical line segment from $1-(\log N)^{-1/2} -iT$ to $1-(\log N)^{-1/2} +iT$, and (iii) 
the horizontal line segment from $1-(\log N)^{-1/2} +iT$ to $c+iT$.  On the horizontal line segment (i) we may bound the integral by
$$ 
\ll \int_{1-(\log N)^{-1/2}}^{c} N^{\sigma} (\log N) |\Gamma(\sigma -iT)| d\sigma \ll N (\log N) e^{-T} \ll N(\log N)^{-100}, 
$$ 
upon using $|\Gamma(\sigma -iT)| \ll T^{\sigma -1/2} e^{-\pi T/2} \ll e^{-T}$.  Naturally the same estimate applies to the integral on the horizontal line segment in (iii).  
As for the vertical line segment (ii), the integral here is 
\begin{align*}
&\ll N^{1- (\log N)^{-1/2}} (\log N) \int_{-T}^{T} |\Gamma( 1-(\log N)^{-1/2} +it)| dt \\
&\ll N (\log N) \exp(-\sqrt{\log N}) \ll N(\log N)^{-100}. 
\end{align*} 
Putting all this together and recalling \eqref{8.1}, we conclude that uniformly for $|z|=1$  
\[
\sum_{v_2(n) = j} \frac{r(n,\psi)}{\tau(\ns)} z^{\Omega(n)}  e^{-n/N} \ll N(\log N)^{-100}.
\]
By Fourier inversion 
\[
\sum_{n \in {\mathcal A}_j(k)} \frac{r(n,\psi)}{\tau(\ns)} = \frac{1}{2\pi} \int_{0}^{2\pi} e^{-ik\theta} d\theta \sum_{v_2(n) = j} \frac{r(n,\psi)}{\tau(\ns)} e^{i\theta \Omega(n)}  e^{-n/N} \ll N(\log N)^{-100}, 
\]
and the proof of Lemma \ref{lem8.2} is complete.
\end{proof}

To state the other two lemmas of this section, we first need to isolate the \emph{genus characters} which play a special role.   The genus characters for a quadratic field are the class group characters that take only the real values $\pm 1$. We will only need to know what these are in the case $d$ prime, in which case the classification is as follows.

\begin{proposition}\label{prop27}
Let $d$ be an odd prime, and let $D$ be the associated fundamental discriminant as in \eqref{3.6}. Let $K = \Q(\sqrt{D})$. 
\begin{enumerate}
\item If $d \equiv 3 \ \md{4}$, so $D = -d$, then there is only one genus character in $\widehat{C}_K$, namely the principal character $\psi_0$. The corresponding class group $L$-function is the Dedekind zeta function of $K$, given by 
$$
L_K(s,\psi_0) = \zeta_K(s) = \zeta(s) L(s,\chi_{D}).
$$ 
\item If $d \equiv 1\  \md{4}$, so $D = -4d$, then there are two genus characters in $\widehat{C}_K$: the principal character $\psi_0$, whose $L$-function is equal to the Dedekind zeta-function of $K$ as above, and a nontrivial genus character $\psi_1$. On prime ideals $\mathfrak{p}$, $\psi_1$ is given by $\psi_1(\mathfrak{p}) = \chi_{-4}(N\mathfrak{p})$ if $\mathfrak{p}$ lies above an odd prime, and $\psi(\mathfrak{p}) = \chi_{d}(2)$ if $\mathfrak{p}$ is the \textup{(}unique, ramified\textup{)} prime ideal above $2$. The corresponding $L$-function is given by $L_K(s,\psi_1) = L(s,\chi_{-4}) L(s,\chi_{d})$.
\end{enumerate}
\end{proposition}
\begin{proof}
There is a bijective correspondence between genus characters of imaginary quadratic fields of discriminant $D$ and factorizations $D = D' \cdot D''$ into fundamental discriminants, with the decomposition $D = 1 \cdot D$ being allowed, and with decompositions different only in the order of $D, D''$ being considered equivalent. See \cite[Chapter 12, Satz 2]{zagier} for a discussion of this, as well as a discussion of how to compute these factorisations in terms of the factorisation of $D$ into \emph{prime discriminants}. For us, the factorisations can easily be computed by hand: if $d \equiv 3\  \md{4}$ then there is only the trivial factorization, whilst when $d \equiv 1\  \md{4}$, in which case $D = -4d$, we additionally have the factorisation in which $D' = -4$ and $D'' = d$. 

In \cite[Chapter 12, Satz 2]{zagier} one may also find the description of a genus character $\psi$ corresponding to a given factorization $D = D' \cdot D''$: it is given on prime ideals by
\[ \psi(\mathfrak{p}) = \left\{ \begin{array}{ll}   \chi_{D'} (N\mathfrak{p}) & (N\mathfrak{p}, D') = 1\\ \chi_{D''}(N\mathfrak{p}) & (N\mathfrak{p}, D'') = 1\end{array} \right.    \] (That this is well-defined is part of the statement.)
We have the \emph{Kronecker factorization} of $L$-functions
\[ L_K(s,\psi) = L(s, \chi_{D'}) L(s, \chi_{D''}).\]
Specialising to the specific case $D' = -4$, $D'' = d$ gives the stated result.
\end{proof}

\begin{lemma} \label{lem8.4}  Let $N$ be large and $k$ be an integer in the range \eqref{1.4range-wide}.  Let $j$ be $0$ or $1$, and let $d$ and ${\widetilde d}$ be two distinct elements of ${\mathcal D}_j$ with $D$ and ${\widetilde D}$ denoting the corresponding fundamental discriminants.  Let $\psi$ and ${\widetilde \psi}$ be characters of the class groups of $K =\Q(\sqrt{D})$ and ${\widetilde K} = \Q(\sqrt{{\widetilde D}})$ respectively.  Then
\[
\sum_{n \in {\mathcal A}_j(k) } \frac{r(n,\psi) r(n, {\widetilde \psi})}{\tau(\ns)^2} e^{-n/N } \ll N(\log N)^{-100}
\]
unless \textup{(i)} $\psi$ and ${\widetilde \psi}$ are the principal characters in their respective class groups, or \textup{(ii)} both $\psi$ and ${\widetilde \psi}$ are the non-principal genus character in their respective class groups \textup{(}and this possibility occurs only in the case $j=1$\textup{)}.  
\end{lemma} 
\begin{proof}   Given two characters $\psi \in {\widehat C}_K$ and ${\widetilde \psi} \in {\widehat C}_{\widetilde K}$, we may find a class group character $\Psi$ on the biquadratic field $L=\Q(\sqrt{D}, \sqrt{{\widetilde D}})$ such that for all unramified primes $p$ 
\[ 
\sum_{\substack{ {\mathfrak P} \subset {\mathcal O}_L \\ N_{L/\Q}(\mathfrak P)=p}} \Psi(\mathfrak P) = \Big( \sum_{\substack{ {\mathfrak p} \subset {\mathcal O}_K \\ N_{K/\Q}(\mathfrak p) = p} }\psi({\mathfrak p}) \Big) 
\Big( \sum_{\substack{ \widetilde{\mathfrak{p}} \subset {\mathcal O}_{\widetilde K} \\ N_{\widetilde{K}/\Q}(\widetilde{\mathfrak{p}}) = p }} {\widetilde \psi}(\widetilde{ \mathfrak p}) \Big) = r(p, \psi) r(p, \widetilde \psi). 
\]
The character $\Psi$ is defined by setting 
\[
\Psi({\mathfrak P}) = \psi(N_{L/K}(\mathfrak P)) {\widetilde \psi} (N_{L/{\widetilde K}} (\mathfrak P))
\]
where $N_{L/K}$ denotes the ideal norm from $L$ to $K$ (and similarly for ${\widetilde K}$) and, by 

Diao \cite[Lemma 6]{diao} we may check that $\Psi$ is non-principal except in the cases (i) and (ii) described in the lemma. (Precisely, Lemma 6 of Diao \cite{diao} 
shows that $\Psi$ can be principal only if $\psi$ and ${\widetilde\psi}$ are genus (or real) characters.  The last remaining case when one of $\psi$ or ${\widetilde \psi}$ 
is principal while the other equals a non-principal genus character is easily checked directly.)   Therefore we may write for any complex number $z$ with $|z|=1$
\[
\sum_{v_2(n) = j} \frac{r(n,\psi) r(n,{\widetilde \psi})}{\tau(\ns)^2} z^{\Omega(n)}n^{-s} = L(s,\Psi)^z {\mathcal G}(s;z,j) 
\] 
where ${\mathcal G}(s;z,j)$ is given by a suitable Euler product which converges absolutely in Re$(s) \geq \frac 34$ and satisfies in that region 
\[
\log {\mathcal G}(s;z,j) \ll W^{1/4}. 
\]
Since the discriminant of $L$ is $\ll \Delta^4$, using Lemma \ref{lem8.1} and arguing exactly as in our proof of Lemma \ref{lem8.2} we establish that 
\[
\sum_{v_2(n) = j} \frac{r(n,\psi) r(n,{\widetilde \psi})}{\tau(\ns)^2} z^{\Omega(n)} e^{-n/N}  \ll N(\log N)^{-100}, 
\]
and then the bound of the lemma follows by an application of Fourier inversion.
\end{proof} 

\begin{lemma} \label{lem8.5} Let $N$ be large and $k$ be an integer in the range \eqref{1.4range-wide}.  Let $j$ be $0$ or $1$, and let $d$ be an element of ${\mathcal D}_j$ with $D$ 
denoting the corresponding fundamental discriminant.  Let $\psi$ and ${\widetilde \psi}$ be two characters of the class group of $K =\Q(\sqrt{D})$, and suppose that 
neither $\psi$ nor $\overline{\psi}$ is equal to ${\widetilde \psi}$.  Then
$$ 
\sum_{n \in {\mathcal A}_j(k) } \frac{r(n,\psi) r(n, {\widetilde \psi})}{\tau(\ns)^2} e^{-n/N } \ll N(\log N)^{-100}. 
$$ 
\end{lemma} 
\begin{proof}  For an unramified prime $p$ we have 
\[
r(p,\psi) r(p, \widetilde{\psi}) = r(p, \psi {\widetilde \psi}) + r(p, \overline{\psi} \widetilde \psi). 
\]
To see this, note that if $p$ is inert then $r(p,\psi)=r(p,\widetilde{\psi})=r(p,\psi \widetilde\psi)=r(p,\overline{\psi} \widetilde \psi)=0$, 
while if $p$ splits as ${\mathfrak p}\overline {\mathfrak p}$ then $r(p,\psi)=\psi(\mathfrak{p}) + \overline{\psi}(\mathfrak{p})$ (and similarly 
for the other quantities) so that the stated relation follows with a little algebra.  
It follows that for any complex number $z$ with $|z|=1$ we may write 
\[
\sum_{v_2(n) = j} \frac{r(n,\psi) r(n,{\widetilde \psi})}{\tau(\ns)^2} z^{\Omega(n)} n^{-s} = L(s,\psi {\widetilde \psi})^z L(s, \overline{\psi} \widetilde{\psi})^z  {\mathcal G}(s;z,j) 
\]
where ${\mathcal G}(s;z,j)$ is given by a suitable Euler product which converges absolutely in Re$(s) \geq \frac 34$ and satisfies in that region 
\[
\log {\mathcal G}(s;z,j) \ll W^{1/4}. 
\]
By hypothesis both $\psi \widetilde{\psi}$ and $\overline{\psi} \widetilde{\psi}$ are non-principal characters of the class group of $K$, and therefore arguing exactly as in Lemma \ref{lem8.2} and  Lemma \ref{lem8.4}, we obtain the lemma.
\end{proof} 

\section{Proof of Proposition \ref{prop5.2}} \label{sec9}

In this section we prove Proposition \ref{prop5.2}. Using \eqref{3.5} we may write 
\[
\sum_{n \in {\mathcal A}_j(k)} \frac{R_D(n)}{\tau(\ns)} e^{-n/N} = \frac{1}{h_K} \sum_{\psi \in {\widehat C}_K} \sum_{n \in {\mathcal A}_j(k)} \frac{r(n,\psi)}{\tau(\ns)} e^{-n/N}.
\]
By Lemma 8.2, the contribution of the non-principal characters is $\ll N(\log N)^{-100}$. Thus 
\begin{equation} 
\label{9.1} 
\sum_{n \in {\mathcal A}_j(k)} \frac{R_D(n)}{\tau(\ns)} e^{-n/N} = \frac{1}{h_K}  \sum_{n \in {\mathcal A}_j(k)} \frac{r(n,\psi_0)}{\tau(\ns)} e^{-n/N} + O\big( N(\log N)^{-100}\big),
\end{equation}
where $\psi_0$ is the trivial character. 

To understand the main term on the right-hand side of \eqref{9.1}, we will again follow Selberg \cite{selberg-54} and introduce for any $z \in \C$ with $|z| = 1$ the Dirichlet series
\begin{equation}
\label{9.2} 
\mathcal{F}(s; z, j) :=  \sum_{v_2(n) = j } \frac{r(n,\psi_0)}{\tau(\ns) } z^{\Omega(n)} n^{-s}, 
\end{equation}
which to begin with converges absolutely for $\sigma =\text{Re}(s) >1$ and defines a holomorphic function there.  As in Selberg's work, we will find that $\mathcal{F}$ can be understood in terms of the complex powers of $\zeta$ and $L$-functions, thereby obtaining an analytic continuation of ${\mathcal F}$ to a wider region. The sum in 
\eqref{9.1} can be expressed in terms of a contour integral involving ${\mathcal F}(s)$, which can then be evaluated using the analytic continuation of ${\mathcal F}$
and an argument involving a Hankel contour. Since we need to keep track of the uniformity in $d$, we give a self-contained account in Appendix \ref{app-A}.

Let us turn to the details. We obtain an analytic continuation of ${\mathcal F}$ to a wider 
region by writing 
\begin{equation} 
\label{9.3} 
{\mathcal F}(s;z,j) = (\zeta(s) L(s,\chi_D))^z {\mathcal G}(s;z,j). 
\end{equation} 
Note that $\zeta(s) L(s,\chi_D)$ is the Dedekind zeta function of the quadratic field $\Q(\sqrt{D})$, and 
by $(\zeta(s) L(s,\chi_D))^z$ we mean $\exp(z\log (\zeta(s) L(s,\chi_D))$, where the logarithm is initially defined in $\sigma >1$ by an absolutely convergent Dirichlet series 
as in \eqref{8.2}.  Thus \eqref{9.3} should be thought of as the definition of the function ${\mathcal G}(s;z,j)$, which is holomorphic in the half plane $\sigma >1$.  We shall shortly see 
that ${\mathcal G}(s;z,j)$ is analytic in $\sigma > \frac 12$ with suitable bounds in that region.  By part (2) of Lemma \ref{lem8.1} we may obtain an analytic continuation of $\log ((s-1)\zeta(s) L(s,\chi_D))$ to the region ${\mathcal R}_0$ with corresponding bounds in the region ${\mathcal R}$.   In this way we obtain a continuation of ${\mathcal F}(s;z,j)$ 
(essentially) to the region ${\mathcal R}$, except that we must omit the real line segment to the left of $s=1$ owing to the logarithmic singularity at $s=1$.

From the multiplicative nature of the definition of $\mathcal{F}(s;z,j)$, in the region $\sigma>1$ we see that ${\mathcal G}(s;z,j)$ is given by 
an Euler product $\prod_{p} {\mathcal G}(s;z,j)$.  We now describe these Euler factors.  If $p>W$ we have 
\begin{equation} 
\label{9.4} 
{\mathcal G}_p(s;z,j) = \big(\sum_{\ell =0}^{\infty} r(p^{\ell}, \psi_0) z^{\ell} p^{-\ell s} \big) \big(1- p^{-s}\big)^z \big( 1- \chi_D(p)p^{-s} \big)^z, 
\end{equation} 
and since $r(p,\psi_0) = 1+ \chi_D(p)$ and $0\leq r(p^\ell,\psi_0)\leq (\ell+1)$ it follows that 
\begin{align} \nonumber 
{\mathcal G}_p(s;z,j) = \big( 1 + (1+\chi_D(p)) z p^{-s} + O(p^{-2\sigma})\big) & \big( 1 - (1+\chi_D(p)) z  p^{-s} + O(p^{-2\sigma})\big)
\\ & = 1 + O(p^{-2\sigma}). \label{9.5}
\end{align}
For $3\leq p\leq W$, from our choice \eqref{5.4}, \eqref{5.5} of $d$, we have  $\chi_D(p)=1$ so that  $r(p^{\ell},\psi_0) = (1 \ast \chi_D)(p^{\ell}) = \ell+1$.  Thus 
\begin{equation} 
\label{9.6} 
{\mathcal G}_p(s;z,j) = \Big(\sum_{\ell =0}^{\infty} \frac{r(p^{\ell}, \psi_0)}{\ell+1}  z^{\ell} p^{-\ell s} \Big)  \big(1- p^{-s}\big)^z \big( 1- p^{-s} \big)^z 
= \big( 1- z p^{-s} \big)^{-1} \big( 1-p^{-s}\big)^{2z}.   
\end{equation}
As we have just seen, for the primes $p\geq 3$ there is no dependence on $j$. By contrast, for $p=2$ the behaviour is different in the two cases $j=0$ (where $\chi_D(2)=1$) and $j=1$ (where 
$\chi_D(2)=0$).  Here we find  
\begin{equation} 
\label{9.7} 
{\mathcal G}_2(s;z,j) = 
\begin{cases} 
(1-2^{-s})^{2z} &\text{ if } j=0\\ 
z2^{-s-1} (1-2^{-s})^z &\text{ if } j=1. 
\end{cases} 
\end{equation}
From \eqref{9.6} and \eqref{9.7} note that for all $p\leq W$ (and uniformly for $|z| = 1$)
\begin{equation} 
\label{9.8} 
{\mathcal G}_p(s;z,j) = 1 + O( p^{-\sigma}). 
\end{equation} 
From \eqref{9.5} and \eqref{9.8} we see that the Euler product $\prod_{p} {\mathcal G}_p(s;z,j)$, which was known initially to converge absolutely in $\Re s =\sigma >1$, in fact converges absolutely for $\sigma >\frac 12$.  Moreover for $\sigma \geq \frac 34$ 
we deduce that 
\begin{equation} 
\label{9.9} 
|{\mathcal G}(s;z,j)| \ll \exp\Big( \sum_{p\leq W} O(p^{-3/4}) \Big) \ll \exp(W^{1/4}). 
\end{equation} 

Now we apply Selberg's method as explained in Appendix \ref{app-A}. Specifically, by \eqref{a.10} and \eqref{a.11}, we obtain

\begin{equation}
\sum_{ v_2(n) = j}  \frac{r(n, \psi_0)}{\tau(\ns)} z^{\Omega(n)} e^{-n/N}   =  {N} \frac{(\log N)^{z-1}}{\Gamma(z)} L(1,\chi_{D})^z {\mathcal G}(1;z,j)  + O_{\eps}\Big( N (\log N)^{\Re z -\frac 32 + \eps}\Big). \label{9.95}
\end{equation}
Applying the above with $z= e^{i\theta}$ for $-\pi \leq \theta \leq \pi$ and applying orthogonality (Fourier inversion), we deduce that 
\begin{align}
 \nonumber
 \sum_{\substack{n \in {\mathcal A}_j(k)}}  &\frac{r(n,\psi_0)}{\tau(\ns)} e^{-n/N}  =  \frac{1}{2\pi} \int_{-\pi}^{\pi} 
  \sum_{v_2(n) = j}  \frac{r(n, \psi_0)}{\tau(\ns)} e^{i\theta \Omega(n)} e^{-n/N} e^{-ik\theta} d\theta \nonumber \\ 
&= \frac{N}{\log N} \frac{1}{2\pi} \int_{-\pi}^{\pi} (L(1,\chi_{D}) \log N)^{e^{i \theta}} \frac{{\mathcal G}(1;e^{i\theta},j)}{\Gamma(e^{i\theta})} e^{-ik\theta} d\theta  +O_{\eps}\big( N (\log N)^{\eps - 1/2} \big). \label{9.10}
\end{align}

We now simplify the main term appearing in \eqref{9.10}.  Since $1/\Gamma(z)$ is entire, uniformly for $\theta \in [-\pi,\pi]$ we have
\begin{equation} 
\label{9.11} 
\frac{1}{\Gamma(e^{i\theta})} = \frac{1}{\Gamma(1)} + O(|e^{i\theta} - 1|)  = 1+ O(|\theta|).
\end{equation} 
Now from its definition in \eqref{9.4} we may see that for $p>W$
 \begin{align*}
 \frac{d}{d\phi} & \log {\mathcal G}_p(1;e^{i\phi}, j)  = i \frac{\sum_{\ell=1}^{\infty} \ell r(p^\ell, \psi_0)e^{i\ell \phi} p^{-\ell}}{\sum_{\ell =0}^{\infty} r(p^\ell, \psi_0)  e^{i\ell \phi}p^{-\ell}} 
 +i e^{i\phi} \log \big( (1 - p^{-1}) (1 - \chi_D(p) p^{-1}) \big) \\
& = i \frac{e^{i\phi} r(p,\psi_0)/p + O(1/p^2)}{1+O(1/p)} - ie^{i\phi}\big( (1+\chi_D(p)) p^{-1}  +O(p^{-2}) \big)
= O(p^{-2}).
 \end{align*}
 Integrating this over $\phi$ from $0$ to $\theta$ we obtain that, for $\theta \in [-\pi, \pi]$, 
 \begin{equation} 
 \label{9.12} 
 \frac{{\mathcal G}_p(1;e^{i\theta},j)}{{\mathcal G}_p(1;1,j)} = \exp\big( O (|\theta| p^{-2}) \big) = 1 + O\big( |\theta| p^{-2} \big). 
 \end{equation}
Similarly, from \eqref{9.6} and \eqref{9.7} it follows that for all $p\leq W$ 
\begin{equation} 
\label{9.13} 
\frac{{\mathcal G}_p(1;e^{i\theta},j)}{{\mathcal G}_p(1;1,j)} = 1 + O \big( |\theta| p^{-1} \big) .
\end{equation}
Multiplying the relations in \eqref{9.12} and \eqref{9.13} over all primes, we conclude that 
\begin{equation} 
\label{9.14} 
\frac{{\mathcal G}(1;e^{i\theta},j)}{{\mathcal G}(1;1,j)} = \exp\Big( O\Big( |\theta| \big( \sum_{p\leq W} p^{-1} + \sum_{p> W} p^{-2} \big) \Big) \Big) 
= 1+ O(|\theta| (\log W)^C), 
\end{equation} 
for some constant $C$ (consider the cases $|\theta| \leq (\log \log W)^{-1}$ and $|\theta| > (\log \log W)^{-1}$ separately).  

For later use, let us record the value of ${\mathcal G}(1;1,j)$.  For any prime $p>W$ one may see using \eqref{9.4} and the identity $\sum_{\ell = 0}^{\infty}\sum_{j = 0}^{\ell}  x^j y^{\ell} = (1 - xy)^{-1}(1 - y)^{-1}$ with $x = \chi_D(p)$ and $y = p^{-s}$ 
that ${\mathcal G}_p(s;1,j) =1$, while for $3\leq p \leq W$ it follows from \eqref{9.6} that ${\mathcal G}_p(s;1,j) = 1-1/p^{s}$, and lastly from 
\eqref{9.7} we see that ${\mathcal G}_2(1;1,j) = 2^{-j-1}(1-2^{-1})$.  Combining these observations, it follows that 
\begin{equation} 
\label{9.15} 
{\mathcal G}(1;1,j)  =2^{-j-1}  \prod_{p\leq W} (1 - 1/p) = 2^{-j-1}  \gamma_W. 
\end{equation}

Using \eqref{9.11}, \eqref{9.14}, and \eqref{9.15} we obtain 
\begin{align}
\label{9.16}
\frac{1}{2\pi} \int_{-\pi}^{\pi}  (L(1,\chi_{D}) \log N)^{e^{i\theta} }& \frac{{\mathcal G}(1;e^{i\theta},j)}{\Gamma(e^{i\theta})} e^{-ik\theta} d\theta 
= 2^{-j-1} \gamma_W \frac{(\log (L(1,\chi_{D} )\log N))^{k}}{k!} \nonumber \\
&+O\Big((\log W)^C \int_{-\pi}^{\pi} (L(1,\chi_D)\log N)^{\cos \theta} |\theta| d\theta \Big),
\end{align}
where the main term arises upon noting that, for $X > 0$, 
\[
\frac{1}{2\pi} \int_{-\pi}^{\pi} X^{e^{i\theta}} e^{-ik\theta} d\theta =\frac{1}{2\pi} \int_{-\pi}^{\pi} 
\sum_{\ell=0}^{\infty} \frac{(\log X)^{\ell}}{\ell!}  e^{i\ell \theta} e^{-ik\theta} d\theta \\
= 
\frac{(\log X)^{k}}{k!}.
\]
Note that $\cos \theta \leq 1- \theta^2/8$ for all $|\theta|\leq \pi$, and that (by the class number formula \eqref{eq9.21})
$L(1,\chi_D) \geq |D|^{-1/2} \geq (\log N)^{-1/2}$ so that $L(1,\chi_D) \log N \geq (\log N)^{1/2}$.  Therefore (recalling that $k_0 = \log \log N$ and that $W = \log \log \log N$)
\[
\int_{-\pi}^{\pi} (L(1,\chi_D)\log N)^{\cos \theta} |\theta| d\theta \leq (L(1,\chi_D) \log N) \int_{-\pi}^{\pi} (\log N)^{-\theta^2/16} |\theta| d\theta 
\ll  k_0^{-1} L(1,\chi_D) \log N,
\]

so that the remainder term in \eqref{9.16} is seen to be 
\begin{equation} 
\label{9.17} 
\ll k_0^{-1} L(1,\chi_{D}) (\log N) (\log W)^C  \ll_{\eps} k_0^{\eps - 1} L(1,\chi_D)\log N.
\end{equation}

Using \eqref{9.16} and \eqref{9.17} in \eqref{9.10} we conclude that 
\begin{equation} 
\label{9.18} 
\sum_{n\in {\mathcal A}_j(k)} \frac{r(n,\psi_0)}{\tau(\ns)} e^{-n/N} = 
\frac{N}{\log N} \frac{\gamma_W}{2^{j+1}} \frac{(\log (L(1,\chi_D)\log N))^k}{k!} + O\big( Nk_0^{\eps - 1} L(1,\chi_D) \big), 
\end{equation}
where we absorbed the error term $O_{\eps}(N(\log N)^{\eps - 1/2})$ in \eqref{9.10} into the error term above using $L(1,\chi_D) \gg_{\eps} |D|^{-\eps} \gg (\log N)^{-\eps}$.

We now claim that for $k$ in the range \eqref{1.4range-wide} and all $x$ with $(\log N)^{-1/2} \leq x \leq k_0^4$ we have 
\begin{equation} 
\label{9.19} 
\Big( \frac{\log (x \log N)}{k_0} \Big)^{k} = \big(1 + \frac{\log x}{k_0}\big)^k =  x \big( 1+ O_{\eps}(k_0^{\eps - 1/3})\big) + O(k_0^{-3}).  
\end{equation} 
To verify the claim, consider the following two cases: (i) when $k_0^{-4} \leq x \leq k_0^4$, and (ii) when $(\log N)^{-1/2} \leq x \leq k_0^{-4}$.  
In case (i) note that $|\log x| = O(\log k_0)$ and so  the LHS of \eqref{9.19} is 
\[ \exp\Big( \Big( \frac{\log x}{k_0} + O_{\eps}(k_0^{\eps - 2}) \Big)(k_0 + O(k_0^{2/3})) \Big) 
= x \exp\big( O_{\eps}(k_0^{\eps - 1/3})\big),
\]
 so that the claim follows here.  In case (ii) note that since $k\geq 3k_0/4$ and we have 
 \[
 \Big(1 +\frac{\log x}{k_0}\Big)^{k} \leq \Big(1+ \frac{\log x}{k_0}\Big)^{3k_0/4} \leq x^{3/4} \leq k_0^{-3}, 
 \]
 so \eqref{9.19} holds in this case also.

Applying \eqref{9.19} with $x=L(1,\chi_D)$ (which satisfies $(\log N)^{-1/2} \leq L(1,\chi_D) \ll \log \log N$) we see that 
\begin{align*}
\frac{(\log (L(1,\chi_D)\log N))^k}{k!} &= \frac{k_0^k}{k!} \Big( L(1, \chi_D) + O_{\eps}\big( k_0^{\eps - 1/3}  L(1,\chi_D)  +k_0^{-3} \big)\Big) \\ 
&= \frac{k_0^k}{k!} L(1,\chi_D) + O_{\eps}\big( k_0^{\eps - 5/6}  L(1,\chi_D)\log N   +k_0^{-3} \log N \big),
\end{align*}
where the last estimate follows using the bound $k_0^k/k! \ll k_0^{-1/2} \log N$, which is a consequence of Stirling's formula. 
Using this in \eqref{9.18} we conclude that 
\begin{equation}
\sum_{n\in {\mathcal A}_j(k)}  \frac{r(n,\psi_0)}{\tau(\ns)} e^{-n/N} = 
\frac{N}{\log N}  \frac{\gamma_W}{2^{j+1}} L(1,\chi_D)\frac{k_0^k}{k!} + O_{\eps}\big( k_0^{\eps - 5/6} NL(1,\chi_D) + k_0^{-3} \log N\big) .\label{9.20}
\end{equation}

Using \eqref{9.20} in \eqref{9.1}, and invoking the class number formula 
\begin{equation}\label{eq9.21} h_K = |D|^{1/2} L(1,\chi_D)/\pi\end{equation}
we obtain (note also that $\gamma_W \gg (\log W)^{-1} \gg_{\eps} k_0^{-\eps}$) 
\[
\frac{|D|^{1/2}}{\pi \gamma_W} \sum_{n\in {\mathcal A}_j(k)} \frac{R_D(n)}{\tau(\ns)} e^{-n/N} =  \frac{1}{2^{j+1}} \frac{N}{\log N} \frac{k_0^k}{k!} 
+ O_{\eps}\big(  Nk_0^{\eps - 5/6} + N L(1,\chi_D)^{-1} k_0^{-2}\big).
\]
Finally, recalling \eqref{partial-sum}, Proposition \ref{prop5.2} follows.

\section{Proof of Proposition \ref{prop5.3}} \label{sec10}

We turn now to the proof of Proposition \ref{prop5.3}. We first dispense with the case when $d\not \equiv {\widetilde d} \ \md 8$ which, recalling the definitions \eqref{5.4} and \eqref{5.5}, can only happen in the case $j = 1$.  Note that if $d \equiv 1 \ \md 8$ then the integers $n$ with $n \equiv 2 \ \md{4}$ that are represented by $x^2+dy^2$ satisfy $n\equiv 2 \ \md 8$, while if $d\equiv 5 \ \md 8$ then such integers $n$ must be $\equiv 6 \ \md 8$.  Thus if $d \not \equiv {\widetilde d} \ \md 8$ and if $n \in \mathcal{A}_1(k)$ (which means that $n \equiv 2 \ \md{4}$) then $n$ cannot be represented by both $x^2 + dy^2$ and $x^2 + \tilde d y^2$. Since both $d, \tilde d$ are $1 \md{4}$, it follows from Lemma \ref{lem3.1} (1) that $R_D(n)R_{\widetilde D}(n) =0$, and so \eqref{5.11} follows.

For the rest of the argument we assume that $d\equiv {\widetilde d} \ \md 8$ with the goal now being to establish \eqref{5.10}.  

Using \eqref{3.5} we see that 
$$ 
\sum_{n\in {\mathcal A}_j(k)} \frac{R_D(n)R_{\widetilde D}(n)}{\tau(\ns)^2} e^{-n/N} = 
\frac{1}{h_K h_{\widetilde K}} \sum_{\substack{\psi \in {\widehat C}_K \\ {\widetilde \psi} \in {\widehat C}_{\widetilde K}}} \sum_{n \in {\mathcal A}_j(k)} 
\frac{r(n,\psi) r(n,\widetilde{\psi})}{\tau(\ns)^2} e^{-n/N}. 
$$ 
We may now use Lemma \ref{lem8.4} to estimate the contribution of all the characters $\psi$ and ${\widetilde \psi}$ apart from when (i) 
both $\psi$ and ${\widetilde \psi}$ are the principal characters in their respective class groups, and (ii) both $\psi$ and ${\widetilde \psi}$ are the non-trivial genus characters in their class groups.  Note that the second case only arises when $j = 1$, and since $d \equiv {\widetilde d} \ \md 8$ we have here (with notation as in Proposition \ref{prop27})
$$
r(n,\psi_1) r(n, \widetilde{\psi}_1) = r(n,\psi_0) r(n, \widetilde{\psi}_0)
$$ 
for all $n$.  To see this, we use Proposition \ref{prop27} and multiplicativity of $r(n,\psi)$ to deduce that $r(n,\psi_1) = \chi_{-4}( n) r(n,\psi_0)$ and similarly 
$r(n, \widetilde{\psi}_1) = \chi_{-4}(n) r(n, \widetilde{\psi}_0)$ for odd 
$n$, so that the stated relation holds for $n$ odd.  Further $r(2^a, \psi_1) = \chi_d(2^a) r(2^a, \psi_0)$ and $r(2^a, \widetilde{\psi}_1) = \chi_{{\widetilde d}}(2^a) r(2^a, \widetilde{\psi}_0)$, and since $\chi_d(2) = \chi_{\widetilde d}(2)$ (because $d\equiv {\widetilde d} \ \md 8$) we see the stated relation for even $n$ as well.     
Thus 
\begin{equation} 
\label{10.1} 
\sum_{n\in {\mathcal A}_j(k)} \frac{R_D(n)R_{\widetilde D}(n)}{\tau(\ns)^2} e^{-n/N} = 
\frac{2^j}{h_K h_{\widetilde K}} M + O\big( N (\log N)^{-100}\big), \end{equation}
where 
\begin{equation} 
\label{10.2} 
 M :=  \sum_{n\in {\mathcal A}_j(k)} \frac{r(n,\psi_0)r(n,{\widetilde \psi}_0)}{\tau(\ns)^2} e^{-n/N}  .
 \end{equation} 
 
 To evaluate the main term $M$ above, we proceed analogously to the previous section using Selberg's method. 
 To highlight the close parallels with the earlier argument, we will use the same notation for the analogous Dirichlet series 
 that arise here.  
 The first step is to consider, for $|z| = 1$, the Dirichlet series
\begin{equation}
\label{10.3} 
\mathcal{F}(s; z, j) :=  \sum_{v_2(n) = j } \frac{r(n,\psi_0)r(n,{\widetilde \psi}_0)}{\tau(\ns)^2 } z^{\Omega(n)} n^{-s}, 
\end{equation}
which converges absolutely for $\sigma =\text{Re }(s) >1$ and defines a holomorphic function there.   As in the 
previous section, we shall obtain an analytic continuation of ${\mathcal F}$ to a wider region by writing 
\begin{equation} 
\label{10.4} 
{\mathcal F}(s;z,j) = (\zeta(s) L(s,\chi_D) L(s,\chi_{\widetilde D}) L(s, \chi_{d{\widetilde d}}))^z {\mathcal G}(s;z,j).
\end{equation} 
Note that $\zeta(s)L(s,\chi_D) L(s,\chi_{\widetilde D}) L(s,\chi_{d\widetilde{d}})$ is the Dedekind zeta function of 
the biquadratic field $L={\mathbb Q}(\sqrt{D}, \sqrt{\widetilde D})$, and 
\[ (\zeta(s)L(s,\chi_D) L(s,\chi_{\widetilde D}) L(s,\chi_{d\widetilde{d}}))^z = \exp(z\log (\zeta(s)L(s,\chi_D) L(s,\chi_{\widetilde D}) L(s,\chi_{d\widetilde{d}}))),\] where the logarithm is initially  defined 
in $\sigma >1$ by an absolutely convergent Dirichlet series as in \eqref{8.2}.  Thus \eqref{10.4} should be thought of as the 
definition of the function ${\mathcal G}(s;z,j)$, which is holomorphic in the half plane $\sigma >1$.  We shall see shortly that ${\mathcal G}(s;z,j)$ 
is analytic in $\sigma > \frac 12$ with suitable bounds in that region.  By Lemma \ref{lem8.1} (2) we may obtain an 
analytic continuation of $\log ((s-1) \zeta(s)L(s,\chi_D) L(s,\chi_{\widetilde D}) L(s,\chi_{d\widetilde{d}}))$ to the region ${\mathcal R}_0$ 
with corresponding bounds in the region ${\mathcal R}$.   In this way we obtain a continuation of ${\mathcal F}(s;z,j)$ essentially to the region ${\mathcal R}$, 
with the caveat that the real line segment to the left of $s=1$ must be omitted owing to the logarithmic singularity at $s=1$.  
 
 From the multiplicative definition of ${\mathcal F}(s;z,j)$, in the region $\sigma> 1$ we see that ${\mathcal G}(s;z,j)$ is 
 given by an  Euler product $\prod_{p} {\mathcal G}_p(s;z,j)$.  We continue as in Section 9 by  describing these Euler 
 factors, and showing that the Euler product converges absolutely in 
$\sigma >\frac 12$ (and assume below that $\sigma >\frac 12$). 

In the case $p>W$ we have 
\[
{\mathcal G}_p(s;z,j) = \big( \sum_{\ell =0}^{\infty} 
r(p^{\ell}, \psi_0) r(p^{\ell},\widetilde{\psi_0}) z^{\ell} p^{-\ell s}\big)  \big( 1-p^{-s}\big)^{z} \big( 1- \chi_D(p)p^{-s}\big)^{z} \big( 1- \chi_{\widetilde D}(p)p^{-s}\big)^z \big( 1-\chi_{d{\widetilde d}}(p) p^{-s}\big)^z.
\]
Since $0\leq r(p^\ell,\psi_0) r(p^{\ell},\widetilde{\psi_0}) \leq (\ell+1)^2$, and $r(p,\psi_0) r(p,\widetilde{\psi_0}) = (1+ \chi_D(p))(1+\chi_{\widetilde D}(p)) 
= 1 +\chi_D(p) + \chi_{\widetilde D}(p) + \chi_{d{\widetilde d}} (p)$ (note that for odd $p$, we have $\chi_D(p) \chi_{\widetilde D}(p) = (\frac{D}{p})(\frac{\widetilde D}{p}) 
= (\frac{D{\widetilde D}}{p}) = (\frac{d{\widetilde d}}{p}) = \chi_{d{\widetilde d}}(p)$), we see that 
\begin{align} 
\label{10.6} 
{\mathcal G}_p(s;z,j) &= \big( 1 + z (1+\chi_D(p)+\chi_{\widetilde D}(p) + \chi_{d {\widetilde d}}(p)) p^{-s} + O(p^{-2\sigma}) \big) \times \nonumber \\
& \qquad \times  \big( 1 - z (1+\chi_D(p)+\chi_{\widetilde D}(p) + \chi_{d {\widetilde d}}(p)) p^{-s} + O(p^{-2\sigma}) \big) \nonumber \\
&= 1+ O(p^{-2\sigma}). 
\end{align}

Turning to the case $3 \leq p \leq W$, recall that $\ns$ is the product of all the prime divisors of $n$ which are $\leq W$, and recall also that $\chi_D(p) = \chi_{\widetilde D}(p)=1$ from 
the definition of $\mathcal{D}_0$,  $\mathcal{D}_1$ (see \eqref{5.4} and \eqref{5.5}).  Thus (as in the last section) we see that 
\begin{equation}
\label{10.7} 
\mathcal{G}_p(s,z,j) = \big(\sum_{\ell =0}^{\infty} z^{\ell} p^{-\ell s} \big) \big( 1-p^{-s} \big)^{4z} 
 = \big(1 - z p^{-s} \big)^{-1} \big( 1- p^{-s}\big)^{4z}.
 \end{equation}
As in the last section, for primes $p\geq 3$ there is no difference between the cases $j=0$ and $j=1$, but at the prime $p = 2$, there is a distinction in 
the definition of ${\mathcal G}_2(s;z,j)$.  Here we have (compare with \eqref{9.7}) 
\begin{equation}
\label{10.8}
 \mathcal{G}_2(s,z,j) =
 \begin{cases} 
 (1-2^{-s})^{4z} &\text{if  } j=0\\
 z2^{-s-2} (1-2^{-s})^{2z} &\text{if } j=1, 
 \end{cases}
 \end{equation} 
 upon noting that when $j=0$ we have $\chi_D(2) = \chi_{\widetilde D}(2) = \chi_{d {\widetilde d}}(2)= 1$ (since $D$, ${\widetilde D}$ and $d{\widetilde d}$ are 
 all $1\ \md 8$), and that when $j=1$ we have $\chi_D(2) = \chi_{\widetilde D}(2) = 0$ and  $\chi_{d\widetilde d}(2) = 1$ (since $d\equiv {\widetilde d} \ \md 8$ so that 
 $d\widetilde d \equiv 1 \ \md 8$).   From \eqref{10.7} and \eqref{10.8} note that for all $p\leq W$ 
 \begin{equation} 
 \label{10.9}  
 {\mathcal G}_p(s;z,j) =  1+ O( p^{-\sigma}).  
 \end{equation} 
 From \eqref{10.6} and \eqref{10.9} we see that the Euler product $\prod_p {\mathcal G}_p(s;z,j)$, which was known initially to converge absolutely for $\sigma >1$, 
 in fact converges absolutely in $\sigma > \frac 12$.  Moreover for $\sigma \geq \frac 34$ we deduce that 
 \begin{equation} 
 \label{10.10} 
 | {\mathcal G}(s;z,j)| \ll \exp\big( \sum_{p\leq W} O(p^{-3/4})\big) \ll \exp(W^{1/4}). 
 \end{equation} 
 
By Selberg's method, specifically by \eqref{a.10} and \eqref{a.12}, we obtain 
\begin{equation}
\label{10.11} 
 \sum_{ v_2(n) = j}   \frac{r(n, \psi_0) r(n,{\widetilde \psi}_0)}{\tau(\ns)^2} z^{\Omega(n)} e^{-n/N}   =  \frac{N}{\log N} \frac{(\ldd \log N)^{z}}{\Gamma(z)}  {\mathcal G}(1;z,j) 
  + O_{\eps}\big( N (\log N)^{\text{Re }z -\frac 32 + \eps}\big)
\end{equation}
where, here and below, 
\[
 \ldd := L(1,\chi_{D}) L(1, \chi_{\tilde D}) L(1, \chi_{d \tilde d}).
 \]
 The main quantity of interest $M$ (see \eqref{10.2}) may be recovered by Fourier inversion, integrating over $z=e^{i\theta}$ with $-\pi \leq \theta \leq \pi$. 
 Thus, analogously to \eqref{9.10} we find 
 $$ 
 M= \frac{N}{\log N} \frac{1}{2\pi} \int_{-\pi}^{\pi} (\ldd \log N)^{e^{i\theta}} \frac{\mathcal G(1;e^{i\theta},j)}{\Gamma(e^{i\theta})} e^{-ik \theta} d\theta 
 +O_{\eps}\big( N (\log N)^{\eps - 1/2}\big). 
 $$ 
 Arguing as in \eqref{9.11}, \eqref{9.12}, \eqref{9.13} we find analogously to \eqref{9.14} 
 $$ 
 \frac{{\mathcal G}(1;e^{i\theta},j)}{\Gamma(e^{i\theta})} = {\mathcal G}(1;1,j) \big( 1+ O (|\theta| (\log W)^C)\big), 
 $$ 
 for a suitable constant $C$.  Using this in our expression for $M$, and arguing as in \eqref{9.16} and \eqref{9.17} we 
 arrive at 
\begin{equation}
\label{10.13}
M = \frac{N}{\log N} \frac{ (\log (\ldd \log N))^k}{k!}  {\mathcal G}(1;1,j)  +O\big( k_0^{\eps - 1} N \ldd \big).
\end{equation}
Here the error term $O_{\eps}(N (\log N)^{-1/2 + \eps})$ has been absorbed into the 
error term above,  since the $L$-values at $1$ are all $\gg |D|^{-\eps} \gg (\log N)^{-\eps}$.

Applying \eqref{9.19} with $x= \ldd$ (which satisfies $(\log N)^{-\eps} \ll \ldd \ll k_0^3$), we 
see that for $k$ in the range \eqref{1.4range-wide} 
\[
\frac{ (\log ( \ldd \log N))^k}{k!} = \ldd
\frac{k_0^{k}}{k!} +(\log N) O_{\eps}\big( k_0^{\eps - 5/6}\ldd + k_0^{-3}\big). 
\]
Using this in \eqref{10.13} we conclude that  
\begin{equation}
\label{10.14}
 M = \frac{N}{\log N} 
{\mathcal G}(1;1,j) \ldd \frac{k_0^k }{k! } + N O_{\eps}\big( k_0^{\eps - 5/6} \ldd +  k_0^{-3} \big).
\end{equation}

From \eqref{10.6}, \eqref{10.7}, and \eqref{10.8} we see that 
\[
{\mathcal G}(1;1,j)  = 2^{-j-1} \gamma_W^3  \prod_{p > W} \big(1 +O(p^{-2})\big)  = 2^{-j-1} \gamma_W^3 \big(1 + O(W^{-1})\big). 
\]
(In fact, rather than use the crude bound \eqref{10.6} one may compute $\mathcal{G}_p(1,1,j) = 1 - \chi_{d\tilde d}(p)p^{-2}$ for $p > W$, but we do not need this.)
Using this together with the class number formula \eqref{eq9.21} and \eqref{10.14} in \eqref{10.1}, \eqref{10.2} we obtain
\begin{align*} 
\sum_{n\in {\mathcal A}_j(k)} &  \frac{R_D(n)R_{{\widetilde D}}(n)}{\tau(f)^2} e^{-n/N} 
=\big(\tfrac{1}{2} + O(W^{-1})\big) \pi^2 \gamma_W^3 |D{\widetilde D}|^{-1/2} L(1,\chi_{d\widetilde {d}}) \frac{N}{ \log N} \frac{k_0^k}{k!} \\
&+ N |D \tilde D|^{-1/2} O_{\eps}\Big( k_0^{\eps - 5/6} L(1,\chi_{d{\widetilde d}}) + 
k_0^{-3}  L(1,\chi_D)^{-1} L(1,\chi_{\widetilde D})^{-1}\Big), 
\end{align*}
where we have absorbed the error term $O(N(\log N)^{-100})$ from \eqref{10.1} into the much larger error term above. 

To complete the proof of Proposition \ref{prop5.3}, we multiply though by $|D \tilde D|^{1/2} /\pi^2 \gamma_W^2$ (noting that any extraneous factors of $\gamma_W$ may be absorbed by $k_0^{\eps}$ terms) and, finally, use \eqref{partial-sum}.

\section{Proof of Proposition \ref{prop5.4}} \label{sec11}

In this final section of the main paper, we establish Proposition \ref{prop5.4}.

Using \eqref{3.5} we see that 
$$ 
\sum_{n\in {\mathcal A}_j(k)} \frac{R_D(n)^2}{\tau(\ns)^2} e^{-n/N} = \frac{1}{h_K^2} \sum_{\psi, \widetilde{\psi} \in {\widehat C}_K} \sum_{n \in {\mathcal A}_j(k)} \frac{r(n,\psi) r(n,\widetilde{\psi})}{\tau(\ns)^2} e^{-n/N}. 
$$ 
Using Lemma \ref{lem8.5}, we may bound the contribution of terms with ${\widetilde \psi} \notin \{ \psi, \overline{\psi}\}$ by $\ll N(\log N)^{-100}$.  It remains to treat the cases when 
${\widetilde \psi} = \psi$ or $\overline{\psi}$.  Note that if ${\mathfrak a}$ is an ideal of norm $n$, then so is $\overline{\mathfrak a}$ and moreover $(n)={\mathfrak a} \overline{\mathfrak a}$.  Therefore $\overline \psi({\mathfrak a}) = {\psi}(\overline{\mathfrak a})$ and it follows that $r(n,\psi) = r(n, \overline{\psi})$ is real valued.  Thus when $\widetilde \psi$ is $\psi$ or ${\overline \psi}$ we have 
   $r(n,\psi) r(n, {\widetilde \psi}) = r(n,\psi)^2$, which is real and non-negative. 
   Collecting the observations so far we find 
\begin{equation} 
\label{11.1} 
\sum_{n\in {\mathcal A}_j(k)} \frac{R_D(n)^2}{\tau(\ns)^2} e^{-n/N} \ll \frac{1}{h_K^2} \sum_{\psi \in {\widehat C}_K}  \sum_{n \in {\mathcal A}_j(k)} \frac{r(n,\psi)^2}{\tau(\ns)^2} e^{-n/N}
+ N (\log N)^{-100}. 
\end{equation} 

The contribution of the real characters $\psi$ (which are the genus characters, and there are at most two of them) to \eqref{11.1} is 
$$ 
\ll \frac{1}{h_K^2}  \sum_{n \in {\mathcal A}_j(k)} \frac{r(n,\psi_0)^2}{\tau(\ns)^2} e^{-n/N} \ll \frac{2^k}{h_K^2} \sum_{n \in {\mathcal A}_j(k)} \frac{r(n,\psi_0)}{\tau(\ns)} e^{-n/N}, 
$$
since $r(n,\psi_0) \leq 2^{\Omega(n)} \leq 2^k$.  
Using \eqref{9.1} and Proposition \ref{prop5.2}, the above quantity is bounded by 
\[
\ll_{\eps} \frac{2^k}{h_K} \frac{\gamma_W}{|D|^{1/2}} \Big(\sum_{n \in \mathcal{A}_j(k)} e^{-n/N} + k_0^{\eps - 5/6} N + k_0^{-2} L(1,\chi_D)^{-1} N \Big). 
\]
By \eqref{partial-sum} and Stirling's formula, this is
\begin{equation}\label{11.2} 
\ll N \frac{2^k}{h_K} \frac{\gamma_W}{|D|^{1/2}} \big(k_0^{-1/2} + k_0^{-2} L(1,\chi_D)^{-1} \big).
\end{equation} 

Now we bound the contribution of the complex characters $\psi$ in \eqref{11.1}.  For a complex character $\psi$, note that 
\begin{equation} 
\sum_{n \in {\mathcal A}_j(k)} \frac{r(n,\psi)^2}{\tau(\ns)^2} e^{-n/N}  \leq \sum_{n} r(n,\psi)^2 e^{-n/N}  = \frac{1}{2\pi i} \int_{c-i\infty}^{c+i\infty} \sum_{n=1}^{\infty} 
 \frac{r(n,\psi)^2}{n^s} N^s \Gamma(s) ds, \label{11.3} 
 \end{equation} 
 where we take $c= 1+ 1/\log N$.  By considering whether $p$ does or does not split in $\Q(\sqrt{D})$, we may check that for any unramified prime $p$  
 $$ 
 r(p,\psi)^2 = 1 + \chi_D(p)  + \sum_{N(\mathfrak p) =p} \psi^2({\mathfrak p}).
 $$ 
 Since $\psi$ is not real, note that $\psi^2$ is not principal. By comparing Euler products we may therefore write 
 $$ 
 \sum_{n=1}^{\infty} \frac{r(n,\psi)^2}{n^s} = \zeta(s) L(s,\chi_D) L(s,\psi^2) {\mathcal G}(s), 
 $$ 
where ${\mathcal G}(s)$ is given by an Euler product which converges absolutely in the region $\Re s > \tfrac 12+ \delta$ and is uniformly bounded in that region.  Moving the line of integration in \eqref{11.3} to the line $\Re s=\frac{3}{4}$, we see that this integral equals 
\begin{equation}
\label{11.4} 
NL(1,\chi_D) L(1,\psi^2) {\mathcal G}(1) + \frac{1}{2\pi i} \int^{3/4 + i \infty}_{3/4 - i \infty} \zeta(s) L(s,\chi_D) L(s,\psi^2) {\mathcal G}(s) N^s \Gamma(s) ds.
\end{equation}
To bound the integral above, we use the convexity bound for $L$-functions (see Chapter 5 of \cite{IwKo}, as well as \eqref{8.4} in the case of $L(s,\psi^2)$) which gives 
$$ 
|\zeta(\tfrac 34+it)| \ll_{\eps} (1+|t|)^{1/8+\eps}, \ \  |L(\tfrac 34+it, \chi_D)| \ll_{\eps} (|D| (1+|t|))^{1/8+\eps}, 
$$
and 
$$
 |L(\tfrac 34+it, \psi^2)| \ll_{\eps} (|D|(1+|t|)^2)^{1/8+\eps}.
$$ 
Noting further that $|{\mathcal G}(\frac 34+it)| \ll 1$, and $|\Gamma(\frac 34+it)| \ll (1+|t|)^{1/4} e^{-\pi |t|/2}$ (see (C.19) of \cite{mv}), we may bound the 
integral  on $\Re s =\frac{3}{4}$ in \eqref{11.4}  by
\[ 
\ll_{\eps} N^{3/4}\int^{\infty}_{-\infty} |D|^{1/4+\eps} (1+|t|) e^{-\pi |t|/2} dt \ll N^{3/4} |D|^{1/2}. 
\]

With this in mind, and referring back to \eqref{11.3}, \eqref{11.4}, it follows that for complex $\psi$ we have
\[ 
\sum_{n \in \mathcal{A}_j(k)} \frac{r(n, \psi)^2}{\tau(\ns)^2} e^{-n/N}  \ll N L(1,\chi_D) L(1, \psi^2) + N^{3/4} |D|^{1/2}
.\]
Now $L(1,\psi^2) \ll (\log |D|)^2$ by \eqref{8.4}, and so we conclude that 
\[ 
\sum_{n \in \mathcal{A}_j(k)} \frac{r(n, \psi)^2}{\tau(\ns)^2} e^{-n/N} \ll N L(1, \chi_D) (\log |D|)^2 + N^{3/4} |D|^{1/2} \ll N L(1,\chi_D) (\log |D|)^2,
\]
 where the last estimate follows since $|D| \ll \log N$ and $L(1,\chi_D) \gg |D|^{-\eps}$.  
 
Thus the contribution of the complex characters $\psi$ to \eqref{11.1} is 
\begin{equation*}
\ll  \frac{1}{h_K^2} h_K  N L(1,\chi_D)  (\log |D|)^2 \ll N |D|^{-1/2} (\log |D|)^2. 
\end{equation*}
Combining this with the contribution of the real characters given in \eqref{11.2} we conclude that 
 $$ 
 \sum_{n \in {\mathcal A}_j(k)} \frac{R_D(n)^2}{\tau(\ns)^2}e^{-n/N} \ll  N \frac{2^k}{h_K} \frac{\gamma_W}{|D|^{1/2}} \big( k_0^{-1/2} + k_0^{-2} L(1,\chi_D)^{-1} \big) 
 + N |D|^{-1/2} (\log |D|)^2. 
 $$ 
 Proposition  \ref{prop5.4} follows upon multiplying through by $|D|/\pi^2 \gamma_W^2$ and using the class number formula together with the trivial bound $\gamma_W^{-2} \ll \log |D|$.

\appendix
\section{Details of Selberg's method}\label{app-A}

In this appendix we supply proofs of the applications of Selberg's method as used in Sections 9 and 10; namely the 
asymptotic formulae \eqref{9.95} and \eqref{10.11}.   Let ${\mathcal F}(s;z,j)$ be defined either as in \eqref{9.2} or \eqref{10.3}, 
and correspondingly let ${\mathcal G}(s;z,j)$ be defined as in \eqref{9.3} or \eqref{10.4}.  Define $f(n)$ to be $r(n,\psi_0)/\tau(\ns)$ 
in the situation of Section 9, and $f(n)$ to be $r(n,\psi_0) r(n,\widetilde{\psi_0})/\tau(\ns)^2$ in the situation of Section 10.   Our goal is 
to obtain the stated asymptotic formulae for 
\begin{equation} 
\label{A.1} 
\sum_{n\in {\mathcal A}_j(k)} f(n) z^{\Omega(n)} e^{-n/N} = \frac{1}{2\pi i} \int_{c-i\infty}^{c+i\infty} {\mathcal F}(s;z,j) \Gamma(s) N^s ds, 
\end{equation} 
where we take $c = 1 +1/\log N$.  

We begin by truncating the integral in \eqref{A.1} to $|\Im s| \leq (\log \log N)^2$.  Note that in both situations under consideration $f(n)$ is  
non-negative and bounded by $\tau(n)^2$ so that 
$$ 
|{\mathcal F}(c+it;z,j)| \leq \sum_{n=1}^{\infty}\tau(n)^2 n^{-c} \ll \prod_{p} \big( 1 + 4p^{-c}  + O(p^{-2})\big) \ll  (\log N)^4.
$$ 
Since $|\Gamma(c+it)| \ll (1+|t|)^{c-1/2} e^{-\pi |t|/2}$ by Stirling's formula, we deduce that the tails of 
the integral in \eqref{A.1} contribute 
$$ 
\int_{|t| \geq (\log \log N)^2} N^c (\log N)^4 (1+|t|)^{c-1/2} e^{-\pi |t|/2} dt \ll  N (\log N)^{-100}.   
$$ 
Thus 
\begin{equation} 
\label{A.2} 
 \sum_{n\in {\mathcal A}_j(k)} f(n) z^{\Omega(n)} e^{-n/N} = \frac{1}{2\pi i} \int_{c-i(\log \log N)^2}^{c+i (\log \log N)^2} {\mathcal F}(s;z,j) \Gamma(s) N^s ds + O\big( 
 N (\log N)^{-100}\big), 
 \end{equation} 
and note that the error term above is negligible compared to the error terms in the formulae \eqref{9.95} and \eqref{10.11} that we are seeking to establish.

To proceed further with evaluating the truncated integral in \eqref{A.2}, we will shift contours using a Hankel or a key-hole type contour.  As in \eqref{w-def-first} denote by ${\mathcal W}$ the 
region
  \[
  \mathcal{W} := \{ 2> \Re s > 1 - 2(\log N)^{-1/2}, |\Im s| < 2(\log \log N)^2\}, 
  \]
 and let ${\mathcal W}_*$ denote the domain ${\mathcal W}$ with the line segment from $1-2 (\log N)^{-1/2}$ to $1$ excised.   In the region ${\mathcal W}_*$ 
 we define $\log (s-1)$ to be the principal branch of the logarithm, taking real values for $s\in (1,\infty]$ and if $s$ lies just above the cut then the argument is $i\pi$, while if $s$ lies just below the cut then the argument is $-i\pi$.  This gives a definition of $(s-1)^{w} = \exp(w \log (s-1))$ 
 (for any complex number $w$), which is holomorphic in ${\mathcal W}_*$.   Now, as in \eqref{9.3} or \eqref{10.4}, we may write 
 ${\mathcal F}(s;z,j) = \zeta_K(s)^z {\mathcal G}(s;z,j)$ where $K$ is either a quadratic (in the case of Section 9) or a biquadratic (in the case of Section 10) 
 field.   Here ${\mathcal G}(s;z,j)$ extends to a holomorphic function in a region containing ${\mathcal W}$, and throughout ${\mathcal W}$ it 
 satisfies the bound $|{\mathcal G}(s;z,j)| \ll \exp(W^{1/4})$ (see \eqref{9.9}   or \eqref{10.10}).   From Lemma \ref{lem8.1} (2) we 
 see that $\zeta(s)^{z}$ extends to a holomorphic function on ${\mathcal W}_*$, and for $s \in {\mathcal W}_*$ with $|\Im(s)| \geq 1$ we have 
 $$ 
 |\zeta_K(s)^z|  \leq \exp( |\log \zeta_K(s)|) \ll (\log N)^{\eps},
 $$ 
 where we used \eqref{8.5} together with $|D_K| \ll \Delta^4 \leq (\log N)^4$.  Finally, again by the second part of Lemma \ref{lem8.1}, we see that 
 $((s-1) \zeta_K(s))^z$ extends to a holomorphic function in ${\mathcal W}$, and satisfies for $s\in {\mathcal W}$ with $|\Im(s)|\leq 1$ 
 $$ 
 |((s-1) \zeta_K(s))^z| \leq \exp( |\log ((s-1) \zeta_K(s))|) \ll (\log N)^{\eps}. 
 $$ 
 
 Synthesizing the remarks above, we conclude that ${\mathcal F}(s;z,j)$ extends holomorphically to ${\mathcal W}_*$, and for 
 $s\in {\mathcal W}_*$ with $|\Im(s)| \geq 1$ satisfies 
 \begin{equation} 
\label{A.4} 
|{\mathcal F}(s;z,j)| \ll (\log N)^{\eps}. 
\end{equation} 
Moreover $(s-1)^z {\mathcal F}(s;z,j)$ extends holomorphically to ${\mathcal W}$, and for $s\in {\mathcal W}$ with $|\Im (s)|\leq 1$ satisfies 
\begin{equation} 
\label{A.5}
|(s-1)^z {\mathcal F}(s;z,j)| \ll (\log N)^{\eps}. 
\end{equation}

We return now to the truncated integral in \eqref{A.2}, which we 
will replace with an integral over the following Hankel-type contour. 

This consists of
\begin{itemize}
\item $\Gamma_1$, the horizontal line segment from $c - i (\log \log N)^2$ to $1 - (\log N)^{-1/2} - i(\log \log N)^2$;
\item $\Gamma_2$, the vertical line segment from $1 - (\log N)^{-1/2} - i(\log \log N)^2$ to $1 - (\log N)^{-1/2} $;
\item $\Gamma_3$, which consists 
 of a path $\Gamma_3^-$ going horizontally from $1-(\log N)^{-1/2}$ to $1-r$ staying just below the line $\Im s = 0$, then a circle 
$\Gamma_3^{\circ}$ of radius $r$ about $s = 1$, then a horizontal path $\Gamma_3^+$ from $1-r$ back to $1 - (\log N)^{-1/2}$ but now staying just above the line $\Im s=0$ (here $r\leq 1/\log N$ is a parameter which we later allow to tend to $0$); 
\item $\Gamma_4$, the vertical line segment from $1 - (\log N)^{-1/2}$ to $1 - (\log N)^{-1/2} + i(\log \log N)^2$;
\item $\Gamma_5$, the  horizontal line segment from $1 - (\log N)^{-1/2} + i(\log \log N)^2$ to $c + i (\log \log N)^2$. 
\end{itemize}

Since the integrand ${\mathcal F}(s;z,j) \Gamma(s)N^s$ is holomorphic in ${\mathcal W}_*$,  we may replace the vertical contour from 
$c - i (\log \log N)^{1/2}$ to $c + i(\log \log N)^2$ by the Hankel-type contour $\Gamma_1 \cup \Gamma_2 \cup \Gamma_3 \cup \Gamma_4 \cup \Gamma_5$. 

(Note that a limiting argument, which we supress, is required to deal with the fact that $\Gamma_3^{\pm}$ lie on the boundary of $\mathcal{W}_*$ rather than within $\mathcal{W}_*$ itself.)  Denote, for $\ell = 1,2,3,4,5$, 
\[  
I_\ell := \frac{1}{2\pi i} \int_{\Gamma_\ell} {\mathcal F}(s;z,j)N^s \Gamma(s) ds.
\]

To estimate the horizontal integrals $I_1$ and $I_5$, we use \eqref{A.4} together with the exponential decay of $|\Gamma(s)|$.  Thus we obtain 
\[
I_1, I_5 \ll \int_{1-(\log N)^{-1/2}}^c N^{\sigma} (\log N)^{\eps} e^{-(\log \log N)^2} d\sigma \ll N (\log N)^{-100}. 
\]
The vertical integrals $I_2$ and $I_4$ are likewise easy to handle.  If $|t| \geq 1$ then \eqref{A.4} gives 
$|{\mathcal F}(1-(\log N)^{-1/2}+it;z,j)| \ll (\log N)^{\eps}$, while if $|t| \leq 1$ then from \eqref{A.5} we deduce 
that $|{\mathcal F}(1-(\log N)^{-1/2}+it;z,j)| \ll (\log N)^{1/2+ \eps}$ (here we take $t$ to be either strictly positive or 
strictly negative, but avoiding point $t=0$).  Combining these estimates with the bound $|\Gamma(1-(\log N)^{-1/2} + it)| \ll e^{-|t|}$ 
we obtain 
\[
I_2, I_4 \ll \int_{0}^{(\log \log N)^2} N^{1-(\log N)^{-1/2}} (\log N)^{1/2+\eps} e^{-|t|} dt \ll N (\log N)^{-100}. 
\]

It remains lastly to consider the integral $I_3$ over the Hankel contour $\Gamma_3$.  Set 
$$ 
{\mathcal H}(s;z,j) = \Gamma(s) (s-1)^z {\mathcal F}(s;z,j). 
$$ 
Consider the circle centered at $1$ with radius $2(\log N)^{-1/2}$.  Since $|\Gamma(s)|$ is bounded in this region, from \eqref{A.5} we see that 
$|{\mathcal H}(s;z,j)| \ll (\log N)^{\eps}$.   Therefore, if $s$ is any point within a circle of radius $(\log N)^{-1/2}$ centered at $1$, we see 
that 
\begin{align*} 
{\mathcal H} & (s;z,j) - {\mathcal H}(1;z,j) = \frac{1}{2\pi i} \int_{|w-1| = 2(\log N)^{-1/2}} {\mathcal H}(w;z,j) \Big( \frac{1}{w-1} - \frac{1}{w-s}\Big) dw \nonumber \\
& \ll \int_{|w-1| = 2(\log N)^{-1/2}} |{\mathcal H}(w;z,j)| \frac{|s-1|}{|(w-1)(w-s)|} |dw| \ll (\log N)^{1/2 + \eps} |s-1| , 
\end{align*}
where we have used Cauchy's formula, and the fact that $|w-s|$ and $|w-1|$ are $\gg (\log N)^{-1/2}$.  Thus 
\begin{align}
\label{A.8}
I_3 &=\frac{1}{2\pi i} \int_{\Gamma_3} (s-1)^{-z} N^s {\mathcal H}(s;z,j) ds \nonumber \\
&= 
\frac{1}{2\pi i} \int_{\Gamma_3} (s-1)^{-z} N^s \Big( {\mathcal H}(1;z,j)  +O\big( |s-1| (\log N)^{1/2+\eps}\big)\Big) ds.
\end{align}

Consider first the error term in \eqref{A.8}.  On the two horizontal parts of $\Gamma_3$, namely $s= \sigma+ 0^{\pm} i$ (depending 
whether we are just above or just below the cut) we have $|(s-1)^{-z}N^s| \ll (1-\sigma)^{-\Re z} N^{\sigma}$, so that these 
integrals contribute 
$$ 
\ll  (\log N)^{1/2 +\eps}  \int_{1-(\log N)^{-1/2}}^{1-r} (1-\sigma)^{1-\Re z} N^{\sigma} d\sigma \ll N (\log N)^{\Re z -3/2+ \eps}. 
$$
Similarly, the (nearly) circular portion of $\Gamma_3$ contributes 
$$ 
\ll N^{1+r} r^{2-\Re z} (\log N)^{1/2 +\eps} \ll N (\log N)^{\Re z -3/2+ \eps},
$$ 
since $r \leq 1/\log N$.  Thus the error term in \eqref{A.8} is $\ll N (\log N)^{\Re z- 3/2+ \eps}$.

Turning to the main term in \eqref{A.8}, we claim that 
\begin{equation}
\label{A.9} 
\frac{1}{2 \pi i }\int_{\Gamma_3} N^s (s - 1)^{-z} ds = \frac{1}{\Gamma(z)} N (\log N)^{z-1} + O(N (\log N)^{-100}).
\end{equation} 
To obtain \eqref{A.9}, denote by $\mathcal{H}$ (the Hankel contour) the contour obtained from $\Gamma_3$ by extending both horizontal parts out to $-\infty$. 
The integral in \eqref{A.9} extended over $\mathcal{H}$ is equal to $\frac{1}{\Gamma(z)} N(\log N)^{z-1}$, as follows from the standard Hankel integral \cite[Theorem II.0.17]{tenenbaum} and a substitution.  Now note that 

\begin{align*}
\int_{\mathcal{H}\setminus \Gamma_3} |N^s (s - 1)^{-z}||ds| & \ll \int_{-\infty}^{1-(\log N)^{-1/2}} N^{\sigma} (1-\sigma)^{-\Re z} d\sigma \\ & \ll N^{1-(\log N)^{-1/2}} (\log N)^{O(1)}, 
\end{align*}
 which is much smaller than $N(\log N)^{-100}$, and thus establishes the claim \eqref{A.9}.   Putting all this together gives
 $$ 
 I_3= {\mathcal H}(1;z,j) N \frac{(\log N)^{z-1}}{\Gamma(z)} + O \big( N (\log N)^{\Re z -3/2 +\eps} \big). 
 $$ 
 
Combining this with our estimates for $I_1$, $I_2$, $I_4$ and $I_5$, from \eqref{A.2} we conclude that 
\begin{equation}\label{a.10}
\sum_{n\in {\mathcal A}_j(k)} f(n) z^{\Omega(n)}e^{-n/N} =  
{\mathcal H}(1;z,j) N \frac{(\log N)^{z-1}}{\Gamma(z)} + O_{\eps} \Big( N (\log N)^{\Re z -3/2 +\eps} \Big). 
\end{equation}
 Finally, in the context of Section 9 note that (using \eqref{9.3}, and since $\lim_{s\to 1} (s-1) \zeta(s)= 1$) 
\begin{equation}\label{a.11}
 {\mathcal H}(1;z,j) = \lim_{s\to 1} \Gamma(s) (s-1)^z {\mathcal F}(s;z,j) = L(1,\chi_D)^z {\mathcal G}(1;z,j), 
\end{equation}
 while in the context of Section 10 (using \eqref{10.4}) 
\begin{equation}\label{a.12} 
 {\mathcal H}(1;z,j) = (L(1,\chi_D) L(1, \chi_{\widetilde D}) L(1,\chi_{d{\widetilde d}}) )^z {\mathcal G}(1;z,j). 
\end{equation}
 This completes our justification of \eqref{9.95} and \eqref{10.11}.

\bibliographystyle{plain}

\begin{thebibliography}{99}

\bibitem{blomer1} V.~Blomer, \emph{Binary quadratic forms with large discriminants and sums of two squareful numbers}, J. Reine Angew. Math. \textbf{569} (2004), 213--234. 

\bibitem{blomer2} V.~Blomer, \emph{Binary quadratic forms with large discriminants and sums of two squareful numbers. {II}}, J. London Math. Soc. (2) \textbf{71} (2005) 69--84.

\bibitem{bg} V.~Blomer and A.~Granville, \emph{Estimates for representation numbers of quadratic forms,}  Duke Math. J. \textbf{135} (2006) no. 2, 261--302.

\bibitem{BF} J.~Bourgain and E.~Fuchs, \emph{A proof of the positive density conjecture for integer Apollonian circle packings}.  J. Amer. Math. Soc. \textbf{24} (2011), 945--967.

\bibitem{cox} D.~A.~Cox, \emph{Primes of the form $x^2+ny^2$}, Wiley-Intersci. Publ.  (1989), xiv+351 pp.

\bibitem{davenport} H.~Davenport, \emph{Multiplicative number theory,} Grad. Texts in Math. \textbf{74}, Springer-Verlag, New York, 2000, xiv+177 pp.

\bibitem{diao}
Y.~Diao, \emph{Density of the union of positive diagonal binary quadratic forms,} Acta Arith. \textbf{207} (2023), no. 1, 1--17.

\bibitem{Fogels1}
E.~Fogels, \emph{On the zeros of {H}ecke's {$L$}-functions. {I}, {II},} Acta Arith. \textbf{7} (1961/2), 87--106 and 131--147.

\bibitem{Fogels2}
E.~Fogels, \emph{\"{U}ber die {A}usnahmenullstelle der {H}eckeschen
  {$L$}-{F}unktionen,} Acta Arith. \textbf{8} (1962/3), 307--309.

\bibitem{GhSa} A.~Ghosh and P.~Sarnak, \emph{Integral points on Markoff type cubic surfaces}.   Invent. Math.  \textbf{229} (2022) 689--749.
  

\bibitem{GS} A.~Granville and K.~Soundararajan, \emph{The distribution of values of {$L(1,\chi_d)$},} Geom. Funct. Anal. \textbf{13} (2003), no. 5, 992--1028.

\bibitem{hv} B.~Hanson and R.~Vaughan, \emph{Density of positive diagonal binary quadratic forms,} Acta Arith. \textbf{193} (2020), no. 1, 1--48.

\bibitem{IwKo} H.~ Iwaniec and E.~Kowalski, \emph{Analytic number theory}, AMS Colloquium Publications \textbf{53},  American Mathematical Society, Providence, RI, 2004, xii+615 pp.

\bibitem{mv} H.~L.~Montgomery and R.~C.~Vaughan, \emph{Multiplicative number theory. I. Classical theory,} Cambridge Stud. Adv. Math. \textbf{97},
Cambridge University Press, Cambridge, 2007, xviii+552 pp.

\bibitem{Nark} W.~Narkiewicz, \emph{Elementary and analytic theory of algebraic numbers.}  Third edition.  Springer Mong. Math. 2004, xii+708 pp. 

\bibitem{selberg-54} A.~Selberg, \emph{Note on a paper by L.~G.~Sathe}, J. Indian Math. Soc. (N.S.) \textbf{18} (1954), 83--87.


\bibitem{tenenbaum} G.~Tenenbaum, \emph{Introduction to analytic and probabilistic number theory}, Grad. Stud. Math. \textbf{163}
American Mathematical Society, Providence, RI, 2015, xxiv+629 pp.

\bibitem{titchmarsh} E.~C.~Titchmarsh, \emph{The theory of functions}, second edition, Oxford University Press, 1952, vi+454 pp. 
  
\bibitem{zagier} D.~B.~Zagier, \emph{Zetafunktionen und quadratische K\"orper. Eine Einf\"uhrung in die h\"ohere Zahlentheorie}, Hochschultext, Springer-Verlag, Berlin-New York, 1981. viii+144 pp.






\end{thebibliography}

\end{document}